\documentclass[11pt]{article}
\usepackage[centertags]{amsmath}
\usepackage{amsfonts}
\usepackage{amssymb}
\usepackage{bm}
\usepackage{makeidx}
\usepackage{newlfont}
\usepackage{amsthm} 
\usepackage{stmaryrd}
\usepackage[]{titletoc}  
\usepackage{mathrsfs}
\usepackage{stmaryrd}
\usepackage{graphicx}  
\usepackage{xypic}
\usepackage{hyperref}
\usepackage[paperwidth=19.5cm,paperheight=27cm,top=2.5cm,bottom=2.5cm,left=2.5cm,right=2.5cm]{geometry}

\newlength{\defbaselineskip}
\newcommand{\setlinespacing}[1]%
           {\setlength{\baselineskip}{#1 \defbaselineskip}}

\theoremstyle{plain}
\newtheorem{thm}{Theorem}[section]
\newtheorem{cor}[thm]{Corollary}
\newtheorem{lem}[thm]{Lemma}
\newtheorem{prop}[thm]{Proposition}
\newtheorem{exam}[thm]{Example}

\newtheorem{Qes}[thm]{Question}

\makeatletter\@addtoreset{equation}{section} \makeatother
\begin{document}

\title {Nevanlinna class, Dirichlet series and Szeg\"{o}'s problem}
\author{
Kunyu Guo \quad Jiaqi Ni \quad Qi Zhou
 }
\date{}
 \maketitle \noindent\textbf{Abstract:}
This paper is associated with Nevanlinna class, Dirichlet series and Szeg\"{o}'s problem in infinitely many variables.
As we will see, there is a natural connection between these topics.
The paper first introduces the Nevanlinna class and the Smirnov class in this context,
and generalizes the classical  theory  in finitely many variables to the infinite-variable setting.
These results  applied to Szeg\"{o}'s problem on Hardy spaces in infinitely many variables.
Moreover, this paper is  also devoted to the study of the correspondence between the Nevanlinna functions and Dirichlet series.

\section{Introduction}

This paper is of three purposes.
The first is to study the function theory of the Nevanlinna class and the Smirnov class in infinitely many variables, the second is to consider Szeg\"{o}'s problem in infinitely many variables, and the third is to discuss the relationship between Dirichlet series and these functions by the Bohr correspondence.

The function theory in infinitely many variables has received attention in recent years, see \cite{AOS2, BBSS, CG, Ni, Ni2}.
We begin with the familiar Hardy spaces.
Let $\mathbb{T}^\infty=\mathbb{T}\times\mathbb{T}\times\cdots$ denote the cartesian product of countably infinitely many unit circles $\mathbb{T}$, equipped with the product topology.
Then $\mathbb{T}^\infty$ is a compact group with the Haar measure $dm_\infty=\frac{d\theta_1}{2\pi}\times\frac{d\theta_2}{2\pi}\times\cdots$.
By a polynomial we mean that it is an analytic polynomial only depending on finitely many complex variables.
Let $\mathcal{P}_\infty$ denote the ring consisting of all polynomials.
For $0<p<\infty$, the Hardy space $H^{p}(\mathbb{T}^{\infty})$ is defined to be the closure of $\mathcal{P}_\infty$ in $L^{p}(\mathbb{T}^{\infty})$. Therefore, when $1\leq p<\infty$, $H^{p}(\mathbb{T}^{\infty})$ is a Banach space with the norm of $L^{p}(\mathbb{T}^{\infty})$, and when $0<p<1$, $H^{p}(\mathbb{T}^{\infty})$ is complete in the metric $d_p(f,g)=\int_{\mathbb{T}^{\infty}}|f-g|^p dm_\infty$. It is clear that the Hardy space $H^{p}(\mathbb{T}^n)$ over the $n$-torus $\mathbb{T}^n$ can be viewed as a closed subspace of $H^{p}(\mathbb{T}^{\infty})$.

Assume $0<p<\infty$. In finite-variable setting, it is known that $H^{p}(\mathbb{T}^n)$ is canonically isometrically isomorphic to the Hardy space $H^{p}(\mathbb{D}^n)$ over the polydisk $\mathbb{D}^n$. We now turn to the infinite-variable setting. Let $\mathbb{D}^\infty=\mathbb{D}\times\mathbb{D}\times\cdots$ be the cartesian product of countably infinitely many open unit disks $\mathbb{D}$. Cole and Gamelin \cite{CG} showed that every function $f\in H^{p}(\mathbb{T}^{\infty})$, by evaluation functional, can be extended to a function $\tilde{f}$ holomorphic on $\mathbb{D}_{2}^\infty=\ell^2\cap\mathbb{D}^\infty$, a domain in the Hilbert space $\ell^2$ of all square-summable sequences.
In particular, for $1\leq p<\infty$, this holomorphic function can be represented by taking Poisson integrals \cite{CG}.
On the other hand, the Hardy space $H^{p}(\mathbb{D}_{2}^\infty)$ over $\mathbb{D}_{2}^\infty$ is defined as follows:
$$H^{p}(\mathbb{D}_{2}^\infty)=\left\{F\;\text{is holomorphic on}\;\mathbb{D}_2^\infty: \|F\|_p^p=\sup_{0<r<1}\int_{\mathbb{T}^\infty}|F_{[r]}|^p dm_\infty<\infty\right\},$$
where
$$F_{[r]}(w)=F(rw_1,\ldots,r^{n}w_n,\ldots),\quad w=(w_1,w_2,\cdots)\in\mathbb{T}^\infty,$$
see \cite{BBSS, DG3}. As the same in the finite-variable setting, for every nonzero function $F\in H^{p}(\mathbb{D}_{2}^\infty)$, the radial limit $F^{*}(w)=\lim_{r\rightarrow1}F_{[r]}(w)$ exists for almost every $w\in\mathbb{T}^{\infty}$, and $\log|F^{*}|\in L^1(\mathbb{T}^\infty)$ \cite{AOS2, BP}. Furthermore, the map $F\mapsto F^{*}$ gives a canonical isometric isomorphism from $H^{p}(\mathbb{D}_{2}^\infty)$ onto $H^{p}(\mathbb{T}^{\infty})$, and its inverse is given by $f\mapsto\tilde{f}$, $f\in H^{p}(\mathbb{T}^{\infty})$, see \cite{AOS2, BBSS, DG3, Ni, Ni2}.

Hardy spaces in infinitely many variables are also closely related to spaces formed by Dirichlet series. Let $\mathcal{P}_D$ be the set of all Dirichlet polynomials $Q(s)=\sum_{n=1}^N a_n n^{-s}$. For $0<p<\infty$ and $Q\in\mathcal{P}_D$, it follows from the almost periodicity of the function $t\mapsto|Q(it)|^p$ that
$$\|Q\|_{p}^{p}=\lim_{T\rightarrow\infty}\frac{1}{2T}\int_{-T}^{T}|Q(it)|^p dt$$
exists, see \cite{Bes}, or \cite[Theorem 1.5.6]{QQ}.
The Hardy-Dirichlet space $\mathcal{H}^p$ is defined to be the completion of $\mathcal{P}_D$ in the metric $\|\cdot\|_p$ \cite{Bay}. Bohr's vision below \cite{Boh} allows us to investigate $\mathcal{H}^p$ via the Hardy space $H^{p}(\mathbb{T}^{\infty})$. Let $\mathbb{N}=\{1,2,\ldots\}$ be the set of positive integers and $p_j$ the $j$-th prime number. With each $n\in\mathbb{N}$ is associated a unique prime factorization $n=p_{1}^{\alpha_1}\cdots p_{k}^{\alpha_k}$, and set $\alpha(n)=(\alpha_1,\ldots,\alpha_k,0,\ldots)$.
For a sequence of complex numbers $\zeta=(\zeta_1,\zeta_2,\ldots)$, write $\zeta^{\alpha(n)}=\zeta_{1}^{\alpha_1}\cdots\zeta_{k}^{\alpha_k}$. The Bohr correspondence
$$\mathcal{B}:\;\sum_{n=1}^N a_n n^{-s}\;\mapsto\sum_{n=1}^{N}a_n\zeta^{\alpha(n)}$$
is an algebraic isomorphism from $\mathcal{P}_D$ onto $\mathcal{P}_\infty$. Then by Birkhoff-Oxtoby theorem \cite[Theorem 6.5.1]{QQ}, for every $Q\in\mathcal{P}_D$, $\|Q\|_{p}^{p}=\int_{\mathbb{T}^{\infty}}|\mathcal{B}Q|^p dm_\infty$, and hence the Bohr correspondence can be extended to an isometric isomorphism from $\mathcal{H}^p$ onto $H^{p}(\mathbb{T}^{\infty})$.

When $p=\infty$, let $H^\infty(\mathbb{T}^\infty)$ be the weak$^{*}$-closure of $\mathcal{P}_\infty$ in $L^{\infty}(\mathbb{T}^\infty)$.
As done in \cite{Aro, CG, Hil}, there is a canonical isometric isomorphism from $H^\infty(\mathbb{T}^\infty)$ onto $H^\infty(\mathbb{D}_{2}^\infty)$, the Banach algebra consisting of all bounded holomorphic functions on $\mathbb{D}_{2}^\infty$, by taking Poisson integrals.
In addition, the Hardy space $H^\infty(\mathbb{T}^\infty)$ can be identified with the Hardy-Dirichlet space $\mathcal{H}^\infty$ by the Bohr correspondence, see works of Hedenmalm, Lindqvist and Seip \cite{HLS}.
For some recent works on the Hardy-Dirichlet spaces $\mathcal{H}^p\;(0<p\leq\infty)$, we refer the  reader to \cite{AOS1, BBSS, BDFMS, BPSSV, BQS, OS}.

The above statements briefly sketch some background material of both  $H^p$ and $\mathcal{H}^p$ in the case $0<p\leq \infty.$  This paper is intended as an attempt to develop the theory of limit function spaces  in two cases of both $H^p$ and $\mathcal{H}^p$ as  $p\rightarrow0^{+}$. We first consider  the case  $H^p$ as  $p\rightarrow0^{+}$, which is  parallel to the finite-variable setting \cite{Ru2}.
For $0<p<\infty$ and $f\in L^{p}(\mathbb{T}^\infty)$, write
$$\|f\|_p=\left(\int_{\mathbb{T}^\infty}|f|^p dm_\infty\right)^{\frac{1}{p}},$$
then when $0<p<q\leq\infty$, $L^{q}(\mathbb{T}^\infty)\subset L^{p}(\mathbb{T}^\infty)$.
As well known, if $f\in L^{r}(\mathbb{T}^\infty)$ for some $0<r\leq\infty$, then $\|f\|_{p}$ tends to $\exp\left(\int_{\mathbb{T}^\infty}\log|f|dm_\infty\right)$ as $p\rightarrow0^{+}$, where $\exp(-\infty)$ is defined to be zero. An observation is that if $f$ is a complex measurable function on $\mathbb{T}^\infty$, then
$\exp\left(\int_{\mathbb{T}^\infty}\log|f|dm_\infty\right)$ is finite if and only if
$\exp\left(\int_{\mathbb{T}^\infty}\log(1+|f|)dm_\infty\right)$ is finite, and the latter is equivalent to
$$\|f\|_0=\int_{\mathbb{T}^\infty}\log(1+|f|)dm_\infty<\infty.$$
So the limit space of $L^{p}(\mathbb{T}^\infty)$ as $p\rightarrow0^{+}$, denoted by $L^{0}(\mathbb{T}^\infty)$, is defined to be the set of all complex measurable functions $f$ on $\mathbb{T}^\infty$ for which $\|f\|_0$ is finite. Then $L^{0}(\mathbb{T}^\infty)$ is a topological vector space with the complete metric $d_0(f,g)=\|f-g\|_0$.
The limit space of $H^{p}(\mathbb{T}^\infty)$ as $p\rightarrow0^{+}$, denoted by $N_{*}(\mathbb{T}^\infty)$, is defined to be the closure of $\mathcal{P}_\infty$ in $L^{0}(\mathbb{T}^{\infty})$, so-called the Smirnov class over $\mathbb{T}^\infty$. It is shown that each function in  $N_{*}(\mathbb{T}^\infty)$ can be uniquely analytically   extended to a domain  $\mathbb{D}_{1}^\infty$ of $\ell^1$, where  $\mathbb{D}_{1}^\infty=\ell^1\cap\mathbb{D}^\infty$ is a domain in the Banach space $\ell^1$ of summable sequences. This  leads to bring in the Smirnov class $N_{*}(\mathbb{D}_1^\infty)$ for holomorphic functions on $\mathbb{D}_{1}^\infty$. Section 2 will be concerned with  the Smirnov class $N_{*}(\mathbb{D}_1^\infty)$ and the Nevanlinna class $N(\mathbb{D}_1^\infty)$, a larger class than  $N_{*}(\mathbb{D}_1^\infty)$.  It is shown that there is a natural correspondence  between the class $N_{*}(\mathbb{T}^\infty)$ and the  class $N_{*}(\mathbb{D}_1^\infty)$. For functions in the  class $N(\mathbb{D}_1^\infty)$, there exists  an analogue of the classical Fatou's theorem, that is, for every function $F\in N(\mathbb{D}_1^\infty)$, the radial limit $F^{*}(w)=\lim_{r\rightarrow1}F(rw_1,\ldots,r^{n}w_{n},\ldots)$ exists  for almost every $w\in\mathbb{T}^{\infty}$. Furthermore, if $F\ne 0$, then $\log|F^{*}|\in L^{1}(\mathbb{T}^\infty)$.

In Section 3, we apply the preceding results to Szeg\"{o}'s problem in infinitely many variables.
Let us first recall Szeg\"{o}'s theorem in one variable case \cite{Gam, Hof, Sze}.
Assume that $K$ is a nonnegative function in $L^1(\mathbb{T})$ with $\log K\in L^1(\mathbb{T})$, and $m_1$ is the normalized Lebesgue measure on $\mathbb{T}$.
Write $\mathbb{C}[z]$ for the ring of all one-variable analytic polynomials, and $\mathbb{C}_0[z]$ for the set of polynomials $q\in\mathbb{C}[z]$ with $q(0)=0$.
Szeg\"{o}'s theorem states that when $1<p<\infty$, the following equality holds:
$$\inf_{q\in\mathbb{C}_0[z]}\int_{\mathbb{T}}|1-q|^pK dm_1=\exp\left(\int_{\mathbb{T}}\log K dm_1\right).$$
In fact, this holds because such $K$ is exactly the modulus of an outer function \cite{Hof}.
Szeg\"{o}'s theorem has a profound influence in many areas, especially in the theory of orthogonal polynomials on the unit circle \cite{Bin, Sim, Sze},
and it can be regarded as the cornerstone of the development of the invariant subspace theory \cite{Gam}.
For more works relating to Szeg\"{o}'s theorem, we refer the reader to \cite{AS, BMW, GZ1, GZ2, Lab, Sar}.
When $p=2$, Nakazi gives an analogous version of Szeg\"{o}'s theorem in two-variable setting \cite{Nak}.
Whereas in the case of infinitely many variables, things become much more complicated.
Let $K$ be a nonnegative function in $L^1(\mathbb{T}^\infty)$ with $\log K\in L^1(\mathbb{T}^\infty)$.
Without loss of generality, we assume that $\|K\|_{1}=1$.
Write
$$S(K)=\inf_{q\in\mathcal{P}_0}\int_{\mathbb{T}^\infty}|1-q|^pK dm_\infty,$$
where $\mathcal{P}_0$ denotes the set of polynomials $q\in\mathcal{P}_\infty$ with $q(0)=0$.
It is easy to verify that $S(K)$ falls into the closed interval $[\exp\left(\int_{\mathbb{T}^\infty}\log K dm_{\infty}\right),1]$.
Naturally, the problem arises which is called Szeg\"{o}'s problem: For which function $K$, $S(K)$ attains the lower bound $\exp\left(\int_{\mathbb{T}^\infty}\log K dm_{\infty}\right)$ or the upper bound $1$?
We will give a complete answer to this problem in Theorem \ref{thm601} and Theorem \ref{thm603} by applying the previous results.
It is worth pointing out  this paper provides a general method which also applies to the finite-variable setting for Szeg\"{o}'s problem.

As in cases of  the Nevanlinna class and the Smirnov class in the infinite-variable setting, there exist
analogies for  Dirichlet series, that is,  we are concerned with the limit Hardy-Dirichlet space in the situation $ p\to 0^+$.
In Section 4,  we introduce the Smirnov-Dirichlet class $\mathcal{N}_{*}$, which is defined to be the completion of $\mathcal{P}_D$ in the metric
$$\|Q\|_0=\lim_{T\rightarrow\infty}\frac{1}{2T}\int_{-T}^{T}\log(1+|Q(it)|)dt,\quad Q\in\mathcal{P}_D.$$
Then $\mathcal{N}_{*}$ can be viewed as the limit space of $\mathcal{H}^p$ when $p\rightarrow0^{+}$.
Moreover, we draw a conclusion that there exists a canonical isometric algebra isomorphism from $\mathcal{N}_{*}$ onto the Smirnov class $N_{*}(\mathbb{D}_{1}^\infty)$ by the Bohr correspondence.
In order to study composition operators on spaces of Dirichlet series, Brevig and Perfekt \cite{BP} defined the class $\mathcal{N}_{u}$ of Dirichlet series $f$ with the abscissa of uniform convergence $\sigma_{u}(f)\leq0$ and
$$\limsup_{\sigma\rightarrow0^{+}}\lim_{T\rightarrow\infty}\frac{1}{2T}\int_{-T}^{T}\log^{+}|f(\sigma+it)|dt<\infty,$$
where $\log^{+}x=\max\{0,\log x\}$ for $x>0$.
This leads to an introduction of  the Nevanlinna-Dirichlet class $\mathcal{N}$, the completion of $\mathcal{N}_{u}$ in the metric
$$\|f\|_0=\limsup_{\sigma\rightarrow0^{+}}\lim_{T\rightarrow\infty}\frac{1}{2T}\int_{-T}^{T}\log(1+|f(\sigma+it)|)dt, \quad f\in\mathcal{N}_{u}.$$
We will prove that the Nevanlinna-Dirichlet class $\mathcal{N}$ can be isometrically embedded into the Nevanlinna class $N(\mathbb{D}_1^\infty)$ by the Bohr correspondence.

\section{The Nevanlinna class and the Smirnov class}

This section is devoted to introducing the Nevanlinna class and the Smirnov class in infinitely many variables.

\subsection{The Nevanlinna class and the Smirnov class over $\mathbb{D}_{1}^\infty$}

We begin with $\infty$-subharmonic functions on $\mathbb{D}_{1}^\infty$.
Motivated by \cite{Ru2}, an upper semicontinuous function $u:\mathbb{D}_{1}^{\infty}\rightarrow[-\infty,\infty)$ is called $\infty$-subharmonic, if $u$ is subharmonic in each variable separately. As the same in single variable,  if $u:\mathbb{D}_{1}^{\infty}\rightarrow[-\infty,\infty)$ is $\infty$-subharmonic, and  $\varphi$ is a non-decreasing convex function on the real line $\mathbb{R}$, then $\varphi\circ u$ is  $\infty$-subharmonic (setting $\varphi(-\infty)=\lim_{t\rightarrow-\infty}\varphi(t)$).

For each function $F$ on $\mathbb{D}_{1}^\infty$ and $0<r<1$, the function $F_{[r]}$ over the infinite torus $\mathbb{T}^\infty$ is defined by
$$F_{[r]}(w)=F(rw_1,\ldots,r^{n}w_n,\ldots),\quad w\in\mathbb{T}^\infty.$$
The following result will be used frequently in this paper.

\begin{prop}\label{thm2101}
If $u$ is $\infty$-subharmonic on $\mathbb{D}_{1}^\infty$, then integrals
$$I_r=\int_{\mathbb{T}^\infty}u_{[r]}dm_\infty\quad(0\leq r<1)$$
increase with $r$. Therefore, if $\{r_n\}_{n=1}^\infty$ is an increasing sequence in $(0,1)$ with $r_n\rightarrow1$ as $n\rightarrow\infty$, then
$\sup_{0<r<1}I_r=\sup_{n\in\mathbb{N}}I_{r_n}$.
\end{prop}

\begin{proof}
Given $0\leq r\leq s<1$, it suffices to prove that $I_r\leq I_s$. For each $n\in\mathbb{N}$, write
$$u_{[r,s,n]}(w)=u(sw_1,\ldots,s^{n}w_n,r^{n+1}w_{n+1},r^{n+2}w_{n+2},\ldots),\quad w\in\mathbb{T}^\infty.$$
Using subharmonicity successively in the first $n$ variables, we obtain
$$I_r\leq\int_{\mathbb{T}^\infty}u_{[r,s,n]}dm_\infty.$$
Note that $s\overline{\mathbb{D}}\times\cdots\times s^n\overline{\mathbb{D}}\times\cdots$ is compact in $\ell^1$, where $\overline{\mathbb{D}}$ denotes the closed unit disk. The upper semicontinuity of $u$ implies that $\{u_{[r,s,n]}\}_{n=1}^\infty$ is uniformly bounded above on $\mathbb{T}^\infty$. Therefore, it follows from Fatou's lemma that
$$I_r\leq\limsup\limits_{n\rightarrow\infty}\int_{\mathbb{T}^\infty}u_{[r,s,n]}dm_\infty
\leq\int_{\mathbb{T}^\infty}\limsup\limits_{n\rightarrow\infty}u_{[r,s,n]}dm_\infty\leq I_s,$$
as desired.
\end{proof}

Let $\overline{\mathbb{D}}^\infty=\overline{\mathbb{D}}\times\overline{\mathbb{D}}\times\cdots$ denote the cartesian product of countably infinitely many closed unit disks $\overline{\mathbb{D}}$, then it is compact with respect to the product topology. The set of all continuous functions on $\overline{\mathbb{D}}^\infty$, denoted by $C(\overline{\mathbb{D}}^\infty)$, is a Banach algebra with the uniform norm. The following corollary is useful.

\begin{cor}\label{thm2101B}
Let $u\in C(\overline{\mathbb{D}}^\infty)$. If $u$ is $\infty$-subharmonic on $\mathbb{D}_{1}^\infty$, then for every $0<r<1$,
$$u(0)\leq\int_{\mathbb{T}^\infty}u_{[r]}dm_\infty\leq\int_{\mathbb{T}^\infty}udm_\infty.$$
\end{cor}

\begin{proof}
By Proposition \ref{thm2101}, for each $r<s<1$, we have
$$u(0)\leq\int_{\mathbb{T}^\infty}u_{[r]}dm_\infty\leq\int_{\mathbb{T}^\infty}u_{[s]}dm_\infty.$$
The desired corollary follows from Lebesgue's dominated convergence theorem by letting $s\rightarrow1$.
\end{proof}

Recall that a complex-valued function $F$ defined on an open subset $V$ of a Banach space $X$  is called holomorphic \cite{Din}, if $F$ satisfies the following two conditions: (i) $F$ is locally bounded. (ii) For each $x_0\in V$ and $x\in X$, the function $F(x_0+zx)$ is holomorphic in parameter $z$ for $x_0+zx\in V$. One easily checks that holomorphic functions are continuous.

Let $F$ be a holomorphic function on $\mathbb{D}_1^\infty$.
Since $\varphi(x)=\max\{0,x\}$ is a non-decreasing convex function on $\mathbb{R}$, $\varphi(\log|F|)=\log^{+}|F|$ is $\infty$-subharmonic.
Hence by Proposition \ref{thm2101},
\begin{equation}\label{equ2102}
\sup\limits_{0<r<1}\int_{\mathbb{T}^\infty}\log^+|F_{[r]}|dm_\infty=\lim_{r\rightarrow1}\int_{\mathbb{T}^\infty}\log^+|F_{[r]}|dm_\infty.
\end{equation}
The Nevanlinna class $N(\mathbb{D}_{1}^\infty)$ over $\mathbb{D}_{1}^\infty$ is defined to be the class of all holomorphic functions $F$ on $\mathbb{D}_{1}^{\infty}$ such that the value of (\ref{equ2102}) is finite.
The following proposition is needed in the sequel.

\begin{prop}\label{thmB2101}
Let $F$ be a nonzero holomorphic function on $\mathbb{D}_1^\infty$.
Then the following statements are equivalent:

(1) $F\in N(\mathbb{D}_1^\infty)$.

(2) $\sup_{0<r<1}\|F_{[r]}\|_0<\infty$.

(3) $\limsup_{r\rightarrow1}\big\|\log|F_{[r]}|\big\|_1<\infty$.
\end{prop}

To prove Proposition \ref{thmB2101}, we need to introduce the infinite polydisk algebra \cite{CG}.
The infinite polydisk algebra $A(\mathbb{D}^\infty)$ is defined to be the closure of $\mathcal{P}_\infty$ in $C(\overline{\mathbb{D}}^\infty)$.
It is worth mentioning that if $F$ is a nonzero holomorphic function on $\mathbb{D}_1^\infty$,
then for every $0<r<1$, $F_{[r]}\in A(\mathbb{D}^\infty)$, and hence $\log|F_{[r]}|\in L^1(\mathbb{T}^\infty)$~\cite{AOS2, BP}.
On the other hand, we can also understand the infinite polydisk algebra by ``boundary-valued functions".
Let $A(\mathbb{T}^\infty)$ be the  norm-closure of $\mathcal{P}_\infty$ in $C(\mathbb{T}^\infty)$, the Banach space of all continuous functions on $\mathbb{T}^\infty$.
Since the infinite polydisk algebra $A(\mathbb{D}^\infty)$ is of the Shilov boundary $\mathbb{T}^\infty$,
one naturally identifies $A(\mathbb{D}^\infty)$ with $A(\mathbb{T}^\infty)$ by the restriction map.

\vskip2mm

\noindent\textbf{Proof of Proposition \ref{thmB2101}.} $(1)\Leftrightarrow(2)$: By inequalities
\begin{equation}\label{equ2102B}
\log^{+}x\leq\log(1+x)\leq\log^{+}x+\log2,\quad x\geq0,
\end{equation}
we see that (1) is equivalent to (2).

$(1)\Leftrightarrow(3)$: It suffices to show (1) implies (3).
Given $0<r_0<1$, since $F$ is a nonzero function, $\log|F_{[r_0]}|\in L^1(\mathbb{T}^\infty)$.
Then for every $r_0<r<1$,
$$\begin{aligned}
\big\|\log|F_{[r]}|\big\|_1&=2\int_{\mathbb{T}^\infty}\log^+|F_{[r]}|dm_\infty-\int_{\mathbb{T}^\infty}\log|F_{[r]}|dm_\infty\\
&\leq2\int_{\mathbb{T}^\infty}\log^+|F_{[r]}|dm_\infty-\int_{\mathbb{T}^\infty}\log|F_{[r_0]}|dm_\infty,
\end{aligned}$$
and hence
$$\limsup_{r\rightarrow1}\big\|\log|F_{[r]}|\big\|_1\leq2\lim_{r\rightarrow1}\int_{\mathbb{T}^\infty}\log^+|F_{[r]}|dm_\infty
-\int_{\mathbb{T}^\infty}\log|F_{[r_0]}|dm_\infty<\infty,$$
as desired. $\hfill\square$

\vskip2mm

 Let $F$ be a function on $\mathbb{D}_{1}^\infty$. For each $n\in\mathbb{N}$, we denote by $A_nF$ Bohr's $n$te Abschnitt of $F$, which is defined by
$$(A_nF)(\zeta)=F(\zeta_1,\ldots,\zeta_n,0,\ldots),\quad \zeta\in\mathbb{D}_{1}^{\infty}.$$
Since $A_nF$ does not depend on variables $\zeta_{j}\;(j>n)$, we may consider it as a function defined on the polydisk $\mathbb{D}^{n}$. For each $w\in\mathbb{T}^\infty$, the slice function of $F$ at $w$ is defined by
$$F_{w}(z)=F(z w_1,\ldots,z^{n}w_n,\ldots),\quad z\in\mathbb{D}.$$
If $F$ is holomorphic on $\mathbb{D}_{1}^{\infty}$, and $w\in\mathbb{T}^\infty$, then for every $n\in\mathbb{N}$,
$$(A_nF)_{w}(z)=F(z w_1,\ldots,z^{n}w_n,0,\ldots),\quad z\in\mathbb{D}$$
is holomorphic on $\mathbb{D}$.
Note that $\{(A_nF)_{w}\}_{n=1}^{\infty}$ converges to $F_{w}$ uniformly on each compact subset of $\mathbb{D}$.
It implies that $F_{w}$ is holomorphic on $\mathbb{D}$.

Let $N(\mathbb{D})$ be the Nevanlinna class over the open unit disk, then for every $F\in N(\mathbb{D})$, the radial limit $F^{*}(\lambda)=\lim_{r\rightarrow1}F(r\lambda)$ exists for almost every $\lambda\in\mathbb{T}$, see \cite[pp.346]{Ru4}.
Now we give the following theorem for which its second part is an analogue of the classical Fatou's theorem.

\begin{thm}\label{thm2102}
Let $F$ be a nonzero function in $N(\mathbb{D}_{1}^\infty)$, then for almost every $w\in\mathbb{T}^{\infty}$, $F_{w}\in N(\mathbb{D})$.
Furthermore, for almost every $w\in\mathbb{T}^{\infty}$, the radial limit
$$F^{*}(w)=\lim_{r\rightarrow1}F_{[r]}(w)$$
exists, and $\log|F^{*}|\in L^1(\mathbb{T}^\infty)$.
\end{thm}

The next corollary  implies that every function $F\in N(\mathbb{D}_{1}^\infty)$ can be uniquely determined by its ``boundary-valued function" $F^{*}$.

\begin{cor}
Let $F\in N(\mathbb{D}_{1}^\infty)$. If $F^*$ vanishes on some subset of $\mathbb{T}^\infty$ with positive measure, then $F=0$.
\end{cor}

For every $z\in\overline{\mathbb{D}}$ and $\zeta\in\overline{\mathbb{D}}^\infty$, set
\begin{equation}\label{equ2102C}
z\star\zeta=(z \zeta_1,\ldots,z^n \zeta_n,\ldots)\in\overline{\mathbb{D}}^\infty.
\end{equation}
To prove Theorem \ref{thm2102}, we need the following lemma, an analogue of \cite[Lemma 3.3.2]{Ru2}.

\begin{lem}\label{thm2102B}
Let $f$ be a nonnegative measurable function on $\mathbb{T}^\infty$. Then
$$\int_{\mathbb{T}^\infty}f(w)dm_\infty(w)=\int_{\mathbb{T}^\infty}\int_{\mathbb{T}}f(\lambda\star w)dm_1(\lambda)dm_\infty(w).$$
\end{lem}

\begin{proof}
Since the Haar measure $m_\infty$ is rotation-invariant, for every $\lambda\in\mathbb{T}$,
$$\int_{\mathbb{T}^\infty}f(w)dm_\infty(w)=\int_{\mathbb{T}^\infty}f(\lambda\star w)dm_\infty(w).$$
Integrating with respect to $\lambda$ over $\mathbb{T}$ and applying Fubini's theorem yield that
$$\int_{\mathbb{T}^\infty}f(w)dm_\infty(w)=\int_{\mathbb{T}}\int_{\mathbb{T}^\infty}f(\lambda\star w)dm_\infty(w)dm_1(\lambda)
=\int_{\mathbb{T}^\infty}\int_{\mathbb{T}}f(\lambda\star w)dm_1(\lambda)dm_\infty(w),$$
which completes the proof.
\end{proof}

\noindent\textbf{Proof of Theorem \ref{thm2102}.} We use ideas from \cite[Theorem 3.3.3]{Ru2} to complete the proof.
For every fixed $w\in\mathbb{T}^{\infty}$, $\log^{+}|F_{w}|$ is subharmonic on $\mathbb{D}$, and hence the integrals
$$I_{w,r}=\int_{\mathbb{T}}\log^{+}|F_{w}(r\lambda)|dm_{1}(\lambda)\quad(0<r<1)$$
increase with $r$. Then the monotone convergence theorem gives
\begin{equation}\label{equ2103}
\begin{aligned}
\int_{\mathbb{T}^\infty}\left(\sup\limits_{0<r<1}I_{w,r}\right)dm_\infty(w)
&=\sup\limits_{0<r<1}\int_{\mathbb{T}^\infty}\int_{\mathbb{T}}\log^{+}|F_{w}(r\lambda)|dm_{1}(\lambda)dm_\infty(w)\\
&=\sup\limits_{0<r<1}\int_{\mathbb{T}^\infty}\int_{\mathbb{T}}\log^{+}|F_{[r]}(\lambda\star w)|dm_{1}(\lambda)dm_\infty(w).
\end{aligned}
\end{equation}
Applying Lemma \ref{thm2102B} to $\log^{+}|F_{[r]}|$, we have
$$\sup\limits_{0<r<1}\int_{\mathbb{T}^\infty}\int_{\mathbb{T}}\log^{+}|F_{[r]}(\lambda\star w)|dm_{1}(\lambda)dm_\infty(w)
=\sup\limits_{0<r<1}\int_{\mathbb{T}^\infty}\log^{+}|F_{[r]}(w)|dm_\infty(w)<\infty.$$
And hence by (\ref{equ2103}), for almost every $w\in\mathbb{T}^{\infty}$, $\sup_{0<r<1}I_{w,r}$ is finite, which implies that $F_{w}\in N(\mathbb{D})$ for such $w$.

Let $E$ be the set of points $w\in\mathbb{T}^{\infty}$ for which $F^{*}(w)=\lim_{r\rightarrow1}F_{[r]}(w)$ exists. Then $E$ is a measurable set. Put
$$S=\left\{w\in\mathbb{T}^\infty: \lim_{r\rightarrow1}F_{w}(r\lambda)\;\text{exists for almost every}\;\lambda\in\mathbb{T}\right\}.$$
Since for almost every $w\in\mathbb{T}^\infty$, $F_{w}\in N(\mathbb{D})$, we see that $S$ is a measurable set, and $m_\infty(S)=1$.
Furthermore, for every fixed $w\in S$, $\lambda\star w\in E$ for almost every $\lambda\in\mathbb{T}$, and thus
$$\int_{\mathbb{T}}\chi_E(\lambda\star w)dm_1(\lambda)=1,$$
where $\chi_E$ denotes the characteristic function of $E$. Moreover,  applying Lemma \ref{thm2102B} to $\chi_E$ gives
$$\begin{aligned}
m_\infty(E)&=\int_{\mathbb{T}^{\infty}}\chi_E(w)dm_\infty(w)\\
&=\int_{\mathbb{T}^{\infty}}\int_{\mathbb{T}}\chi_E(\lambda\star w)dm_1(\lambda)dm_\infty(w)\\
&\geq\int_{S}\int_{\mathbb{T}}\chi_E(\lambda\star w)dm_1(\lambda)dm_\infty(w)\\
&=\int_{S}1dm_\infty(w)\\
&=1,
\end{aligned}$$
forcing $m_\infty(E)=1$. Then for almost every $w\in\mathbb{T}^{\infty}$, $F^{*}(w)=\lim_{r\rightarrow1}F_{[r]}(w)$ exists, and it follows from Proposition \ref{thmB2101} and Fatou's lemma that
$$\big\|\log|F^{*}|\big\|_1\leq\liminf_{r\rightarrow1}\big\|\log|F_{[r]}|\big\|_1<\infty,$$
which implies that $\log|F^{*}|\in L^1(\mathbb{T}^\infty)$. $\hfill\square$

\vskip2mm

As done in the finite-variable setting, we characterize Nevanlinna functions via the $\infty$-harmonic majorant.
Motivated by \cite{Ru2}, a continuous function $u:\mathbb{D}_1^\infty\rightarrow\mathbb{C}$ is called $\infty$-harmonic, if $u$ is harmonic in each variable separately.
The harmonic Hardy space $h^1(\mathbb{D}_1^\infty)$ over $\mathbb{D}_1^\infty$ is defined to be
$$h^1(\mathbb{D}_1^\infty)=\left\{F\;\text{is}\; \infty\text{-harmonic on}\;\mathbb{D}_1^\infty: \|F\|_{h}=\sup_{0<r<1}\int_{\mathbb{T}^\infty}|F_{[r]}|dm_\infty<\infty\right\}.$$
As in finite-variable cases, each function in $h^1(\mathbb{D}_1^\infty)$ can be represented by the Poisson integral of a complex regular Borel measure on $\mathbb{T}^\infty$.
To show this, we recall some notions of Poisson kernels in infinitely many variables, see \cite{CG}.
For each $\zeta\in\mathbb{D}_{1}^\infty$, the Poisson kernel at $\zeta$ is defined to be
$$\mathbf{P}_\zeta(w)=\prod_{n=1}^{\infty}\mathrm{P}_{\zeta_n}(w_n),\quad w\in\mathbb{T}^\infty,$$
where
$$\mathrm{P}_{\zeta_{n}}(w_{n})=\frac{1-|\zeta_n|^2}{|\zeta_n-w_n|^2},\quad n\in\mathbb{N}$$
are Poisson kernels for the unit disk.
It is easy to prove that when $\zeta\in\mathbb{D}_{1}^\infty$, $\mathbf{P}_\zeta$ is continuous on $\mathbb{T}^\infty$.
Furthermore, if $u\in C(\overline{\mathbb{D}}^\infty)$ is $\infty$-harmonic on $\mathbb{D}_1^\infty$, then for every $\zeta\in\mathbb{D}_{1}^\infty$,
\begin{equation}\label{equB21012}
u(\zeta)=\int_{\mathbb{T}^\infty}\mathbf{P}_\zeta udm_\infty.
\end{equation}
Let $M(\mathbb{T}^\infty)$ denote the Banach space of all complex regular Borel measures $\mu$ on $\mathbb{T}^\infty$ with the norm $\|\mu\|_M=|\mu|(\mathbb{T}^\infty)$.
For every $\mu\in M(\mathbb{T}^\infty)$, set
$$\mathrm{P}[d\mu](\zeta)=\int_{\mathbb{T}^\infty}\mathbf{P}_\zeta d\mu,\quad\zeta\in\mathbb{D}_1^\infty.$$
Then $\mathrm{P}[d\mu]$ is $\infty$-harmonic on $\mathbb{D}_1^\infty$, and by Fubini's theorem, $\|\mathrm{P}[d\mu]\|_h\leq\|\mu\|_M<\infty$,
and hence $\mathrm{P}[d\mu]\in h^1(\mathbb{D}_1^\infty)$.
Moreover, $\mathrm{P}[d\mu]=0$ implies $\mu=0$.

\begin{prop}
The Poisson integral $\mu\mapsto\mathrm{P}[d\mu]$ establishes an isometric isomorphism from $M(\mathbb{T}^\infty)$ onto $h^1(\mathbb{D}_1^\infty)$, and hence $h^1(\mathbb{D}_1^\infty)$ is a Banach space.
\end{prop}

\begin{proof}
By the arguments above, the map $\mu\mapsto\mathrm{P}[d\mu]$ is a one-to-one contraction.
On the other hand, for every $F\in h^1(\mathbb{D}_1^\infty)$ and $0<r<1$, write $d\mu_r=F_{[r]}dm_\infty$.
Then it follows from Banach-Alaoglu theorem that there exist a sequence $r_n\rightarrow1\;(n\rightarrow\infty)$ and a complex Borel measure $\mu_F$ on $\mathbb{T}^\infty$ such that $\{\mu_{r_n}\}_{n=1}^\infty$ converges to $\mu_F$ in the weak$^{*}$-topology, in the dual space of $C(\mathbb{T}^\infty)$.
Combining this fact with (\ref{equB21012}) shows that for every $\zeta\in\mathbb{D}_1^\infty$,
$$\mathrm{P}[d\mu_F](\zeta)=\int_{\mathbb{T}^\infty}\mathbf{P}_\zeta d\mu_F
=\lim_{n\rightarrow\infty}\int_{\mathbb{T}^\infty}\mathbf{P}_\zeta F_{[r_n]}dm_\infty=\lim_{n\rightarrow\infty}F_{[r_n]}(\zeta)=F(\zeta),$$
which implies that $F=\mathrm{P}[d\mu_F]$, and thus the map $\mu\mapsto\mathrm{P}[d\mu]$ is onto.
Furthermore, we have
$$\|\mu_F\|_M\leq\liminf_{n\rightarrow\infty}\|\mu_{r_n}\|_M\leq\sup_{0<r<1}\int_{\mathbb{T}^\infty}|F_{[r]}|dm_\infty=\|F\|_h,$$
and hence the map $\mu\mapsto\mathrm{P}[d\mu]$ is an isometry.
\end{proof}

Recall that $\mathbf{P}_\zeta dm_\infty$ is a Jensen measure for $\zeta\in\mathbb{D}_{1}^\infty$ with respect to $A(\mathbb{D}^\infty)$, that is,
\begin{equation}\label{equB21011}
\log|F(\zeta)|\leq\int_{\mathbb{T}^\infty}\mathbf{P}_\zeta \log|F|dm_\infty,\quad F\in A(\mathbb{D}^\infty),
\end{equation}
see \cite{CG}.
This says that
$$G(\zeta)=\int_{\mathbb{T}^\infty}\mathbf{P}_\zeta \log|F|dm_\infty,\quad\zeta\in\mathbb{D}_{1}^\infty$$
is an $\infty$-harmonic majorant of $\log|F|$.
The next proposition gives a characterization of functions in $N(\mathbb{D}_{1}^\infty)$ via $\infty$-harmonic majorant.

\begin{prop}\label{thmB2103}
Let $F$ be a holomorphic function on $\mathbb{D}_{1}^\infty$. Then the following statements are equivalent:

(1) $F\in N(\mathbb{D}_{1}^\infty)$.

(2) $\log|F|$ has an $\infty$-harmonic majorant in $h^1(\mathbb{D}_1^\infty)$. That is, there is a function $H\in h^1(\mathbb{D}_1^\infty)$ such that $\log|F|\leq H$ on $\mathbb{D}_{1}^\infty$.
\end{prop}

\begin{proof}
$(1)\Rightarrow(2)$: For every $0<r<1$, set $d\nu_r=\log|F_{[r]}|dm_\infty$.
Then it follows from (3) of Proposition \ref{thmB2101} and Banach-Alaoglu theorem that there is a sequence $r_n\rightarrow1\;(n\rightarrow\infty)$ and a complex Borel measure $\nu$ on $\mathbb{T}^\infty$, such that $\{\nu_{r_n}\}_{n=1}^\infty$ converges to $\nu$ in the weak$^{*}$-topology, in the dual space of $C(\mathbb{T}^\infty)$.
Since for every $n\in\mathbb{N}$, $F_{[r_n]}\in A(\mathbb{D}^\infty)$, we see from (\ref{equB21011}) that
$$\log|F_{[r_n]}(\zeta)|\leq\int_{\mathbb{T}^\infty}\mathbf{P}_\zeta \log|F_{[r_n]}|dm_\infty,\quad\zeta\in\mathbb{D}_1^\infty.$$
Letting $n\rightarrow\infty$ yields that
$$\log|F(\zeta)|\leq\int_{\mathbb{T}^\infty}\mathbf{P}_\zeta d\nu,\quad\zeta\in\mathbb{D}_1^\infty,$$
and the right side of this inequality gives an $\infty$-harmonic majorant of $\log|F|$ in $h^1(\mathbb{D}_1^\infty)$.

$(2)\Rightarrow(1)$: By assumptions, we have
$$\sup\limits_{0<r<1}\int_{\mathbb{T}^\infty}\log^+|F_{[r]}|dm_\infty\leq
\sup\limits_{0<r<1}\int_{\mathbb{T}^\infty}H^{+}_{[r]}dm_\infty\leq
\sup\limits_{0<r<1}\int_{\mathbb{T}^\infty}|H_{[r]}|dm_\infty<\infty,$$\
which gives $F\in N(\mathbb{D}_{1}^\infty)$.
\end{proof}

We next consider the metric in the class $N(\mathbb{D}_{1}^\infty)$.
Writing $\varphi(x)=\log(1+\mathrm{e}^x)$, then it is a non-decreasing convex function on $\mathbb{R}$. Therefore, for every $F\in N(\mathbb{D}_{1}^\infty)$, $\varphi(\log|F|)=\log(1+|F|)$ is $\infty$-subharmonic.
Hence by Proposition \ref{thm2101}, integrals
$$\int_{\mathbb{T}^\infty}\log(1+|F_{[r]}|)dm_\infty\quad(0<r<1)$$
increase with $r$.
As done in \cite{Dav, SS, Yan} for finite-variable cases, we define
$$\|F\|_0=\sup_{0<r<1}\int_{\mathbb{T}^\infty}\log(1+|F_{[r]}|)dm_\infty
=\lim_{r\rightarrow1}\int_{\mathbb{T}^\infty}\log(1+|F_{[r]}|)dm_\infty,\quad F\in N(\mathbb{D}_{1}^\infty),$$
and
$$d_{0}(F,G)=\|F-G\|_0,\quad F,G\in N(\mathbb{D}_{1}^\infty),$$
then $d_{0}$ is a translation-invariant metric on the space $N(\mathbb{D}_{1}^\infty)$. In fact, we have
\begin{prop}\label{thm2105}
$N(\mathbb{D}_{1}^\infty)$ is complete with respect to the metric $d_0$.
\end{prop}

To prove Proposition \ref{thm2105}, we need Lemma \ref{thm2105B}, which implies that the evaluation functional at $\zeta\in\mathbb{D}_1^\infty$ on $N(\mathbb{D}_{1}^\infty)$ is continuous, and will be used frequently in the sequel.

Recall that for any $\zeta\in\mathbb{D}^\infty$, the M$\ddot{\mathrm{o}}$bius map over the infinite-dimensional polydisk is defined as
$$\Phi_{\zeta}(w)=\left(\frac{\zeta_{1}-w_{1}}{1-\overline{\zeta}_{1}w_{1}},
\frac{\zeta_{2}-w_{2}}{1-\overline{\zeta}_{2}w_{2}},\ldots\right),\quad w\in\overline{\mathbb{D}}^\infty.$$
Then $\Phi_{\zeta}$ maps $\mathbb{D}^\infty$ and $\mathbb{T}^\infty$ onto themselves, respectively. Moreover, a direct verification shows that $\Phi_{\zeta}$ maps $\mathbb{D}_{1}^\infty$ onto itself if and only if $\zeta\in\mathbb{D}_1^\infty$. When $\zeta\in\mathbb{D}_1^\infty$, an  observation is that for every nonnegative continuous function $f$ on $\mathbb{T}^\infty$,
\begin{equation}\label{equ2103B}
\int_{\mathbb{T}^\infty}f\circ\Phi_{\zeta}dm_\infty=\int_{\mathbb{T}^\infty}f\mathbf{P}_\zeta dm_\infty\leq\|\mathbf{P}_\zeta\|_\infty\int_{\mathbb{T}^\infty}fdm_\infty.
\end{equation}

\begin{lem}\label{thm2105B}
If $\zeta\in\mathbb{D}_{1}^\infty$, then for every $F\in N(\mathbb{D}_{1}^\infty)$,
$$\log(1+|F(\zeta)|)\leq\|\mathbf{P}_\zeta\|_\infty\|F\|_0.$$
\end{lem}

\begin{proof}
For each $0<r<1$, $\log(1+|F_{[r]}\circ\Phi_{\zeta}|)\in C(\overline{\mathbb{D}}^\infty)$ is $\infty$-subharmonic on $\mathbb{D}_{1}^\infty$. By Corollary \ref{thm2101B},
\begin{equation}\label{equ2103C}
\log(1+|F_{[r]}(\zeta)|)=\log(1+|(F_{[r]}\circ\Phi_{\zeta})(0)|)\leq\int_{\mathbb{T}^\infty}
\log(1+|F_{[r]}\circ\Phi_{\zeta}|)dm_\infty.
\end{equation}
Applying the inequality (\ref{equ2103B}) to $\log(1+|F_{[r]}|)$, we have
$$\int_{\mathbb{T}^\infty}\log(1+|F_{[r]}\circ\Phi_{\zeta}|)dm_\infty
\leq\|\mathbf{P}_\zeta\|_\infty\int_{\mathbb{T}^\infty}\log(1+|F_{[r]}|)dm_\infty.$$
Combining this with (\ref{equ2103C}) yields that
$$\log(1+|F_{[r]}(\zeta)|)\leq\|\mathbf{P}_\zeta\|_\infty\int_{\mathbb{T}^\infty}\log(1+|F_{[r]}|)dm_\infty.$$
Letting $r\rightarrow1$ gives the desired conclusion.
\end{proof}

For every $0<r<1$ and $M>0$, set
\begin{equation}\label{equ2103A}
V_{r,M}=\left\{\zeta\in\ell^1: \|\zeta\|_1<M\;\text{and for each}\;n\in\mathbb{N},\;|\zeta_n|<r\right\},
\end{equation}
then $V_{r,M}$ is a domain in $\ell^1$, and $\{\mathbf{P}_\zeta\}_{\zeta\in V_{r,M}}$ is a bounded set in $L^\infty(\mathbb{T}^\infty)$.
We now present the proof of Proposition \ref{thm2105}.

\vskip2mm

\noindent\textbf{Proof of Proposition \ref{thm2105}.} Suppose that $\{F_{n}\}_{n=1}^\infty$ is a Cauchy sequence in $N(\mathbb{D}_{1}^\infty)$. By Lemma \ref{thm2105B}, for each $\zeta\in\mathbb{D}_{1}^\infty$ and $k,l\in\mathbb{N}$,
$$\log(1+|(F_k-F_l)(\zeta)|)\leq\|\mathbf{P}_\zeta\|_\infty\|F_k-F_l\|_0.$$
Note that for each $0<r<1$ and $M>0$, $\{\mathbf{P}_\zeta\}_{\zeta\in V_{r,M}}$ is bounded in $L^{\infty}(\mathbb{T}^\infty)$. It follows that $\{F_n\}_{n=1}^\infty$ converges uniformly to a holomorphic function on each $V_{r,M}$, and hence there exists a holomorphic function $F$ on $\mathbb{D}_{1}^\infty$, such that $\{F_n\}_{n=1}^\infty$ converges to $F$ uniformly on each $V_{r,M}$. It remains to prove $F\in N(\mathbb{D}_{1}^\infty)$ and $\|F_k-F\|_0\rightarrow0$ as $k\rightarrow\infty$. Given $\varepsilon>0$, choose $N$ large enough such that for $k,l>N$, $\|F_k-F_l\|_0<\varepsilon$. Then for each $0<r<1$,
$$\int_{\mathbb{T}^\infty}\log(1+|(F_k-F_l)_{[r]}|)dm_\infty<\varepsilon.$$
Note that $\{F_l\}_{l=1}^\infty$ converges to $F$ uniformly on $r\overline{\mathbb{D}}\times\cdots\times r^n\overline{\mathbb{D}}\times\cdots$. Letting $l\rightarrow\infty$ yields that when $k>N$,
\begin{equation}\label{equ2103E}
\int_{\mathbb{T}^\infty}\log(1+|(F_k-F)_{[r]}|)dm_\infty\leq\varepsilon,
\end{equation}
and hence for all $0<r<1$,
$$\int_{\mathbb{T}^\infty}\log(1+|F_{[r]}|)dm_\infty\leq\varepsilon+\int_{\mathbb{T}^\infty}\log(1+|(F_k)_{[r]}|)dm_\infty
\leq\varepsilon+\|F_k\|_0<\infty.$$
This shows $F\in N(\mathbb{D}_{1}^\infty)$.
Furthermore, by (\ref{equ2103E}), we have $\|F_k-F\|_0\leq\varepsilon$, which means that $N(\mathbb{D}_{1}^\infty)$ is complete with respect to $d_0$. $\hfill\square$

\vskip2mm

An important subclass of $N(\mathbb{D}_{1}^\infty)$ is the Smirnov class $N_{*}(\mathbb{D}_{1}^\infty)$, which consists of all functions $F\in N(\mathbb{D}_{1}^\infty)$ for which $\{\log^{+}|F_{[r]}|\}_{0<r<1}$ forms a uniformly integrable family.
When $F\in N(\mathbb{D}_{1}^\infty)$, (\ref{equ2102B}) implies that $F\in N_{*}(\mathbb{D}_{1}^\infty)$ if and only if $\{\log(1+|F_{[r]}|)\}_{0<r<1}$ is uniformly integrable over $\mathbb{T}^\infty$.

By Fatou's lemma, for every function $F\in N(\mathbb{D}_{1}^\infty)$, using the metric in $L^0(\mathbb{T}^\infty)$,
$$\|F^{*}\|_0\leq\lim\limits_{r\rightarrow1}\|F_{[r]}\|_0.$$
This inequality inspires the following characterization of functions in $N_{*}(\mathbb{D}_{1}^\infty)$.
For finite-variable cases, see \cite{Dav, SS, Yan}.

\begin{prop}\label{thm2103}
Let $F$ be a function in $N(\mathbb{D}_{1}^\infty)$. Then the following statements are equivalent:

(1) $F\in N_{*}(\mathbb{D}_{1}^\infty)$.

(2) $\|F^{*}\|_0=\lim_{r\rightarrow1}\|F_{[r]}\|_0$.

(3) $\lim_{r\rightarrow1}\|F_{[r]}-F^{*}\|_0=0$.

(4) $\lim_{r\rightarrow1}\int_{\mathbb{T}^\infty}\log^{+}|F_{[r]}|dm_\infty=\int_{\mathbb{T}^\infty}\log^{+}|F^{*}|dm_\infty$.
\end{prop}

It is worth mentioning that when $F$ only depends on single variable, the equality (2) in Proposition \ref{thm2103} is another common definition for the Smirnov class $N_{*}(\mathbb{D})$.

To prove Proposition \ref{thm2103}, the the general Lebesgue's dominated convergence theorem \cite[Theorem 2.8.8]{Bog} is needed.

\begin{lem}[General Lebesgue's dominated convergence theorem]\label{thm2103B}
Let $(X,\mathcal{M},\mu)$ be a measure space and $\{f_n\}_{n=1}^\infty$, $\{g_n\}_{n=1}^\infty$ two sequences of measurable functions on $X$ that converge to $f$, $g$ almost everywhere on $X$, respectively.
Assume that for every $n\in\mathbb{N}$, $g_n\geq0$, and $|f_n|\leq g_n$ on $X$.
If $g\in L^1(X,\mu)$ and
$$\lim_{n\rightarrow\infty}\int_{X}g_n d\mu=\int_{X}g d\mu,$$
then
$$\lim_{n\rightarrow\infty}\int_{X}f_n d\mu=\int_{X}f d\mu.$$
Furthermore, this conclusion remain valid if ``convergence almost everywhere" is replaced by ``convergence in measure".
\end{lem}

The following lemma immediately follows.

\begin{lem}\label{thm2104}
Let $0\leq p<\infty$ and $\{h_{n}\}_{n=1}^\infty$ be a sequence in $L^{p}(\mathbb{T}^{\infty})$ that converges to $h\in L^{p}(\mathbb{T}^{\infty})$ almost everywhere on $\mathbb{T}^{\infty}$. Then $\|h_{n}-h\|_p\rightarrow0$ if and only if $\|h_{n}\|_p\rightarrow\|h\|_p$ as $n\rightarrow\infty$.
\end{lem}

\begin{proof}
We prove this lemma by using the general Lebesgue's dominated convergence theorem.
It suffices to show that $\|h_{n}\|_p\rightarrow\|h\|_p$ implies $\|h_{n}-h\|_p\rightarrow0$.
Write
$$f_n=\begin{cases}\log(1+|h_n-h|), & p=0,\\
|h_n-h|^p, & p>0,\end{cases}
\quad\quad g_n=\begin{cases} \log(1+|h_n|)+\log(1+|h|), & p=0,\\
2^p(|h_n|^p+|h|^p), & p>0.\end{cases}$$
For every $n\in\mathbb{N}$, it is clear that $0\leq f_n\leq g_n$ on $\mathbb{T}^{\infty}$.
Applying Lemma \ref{thm2103B} to sequences $\{f_n\}_{n=1}^\infty$ and $\{g_n\}_{n=1}^\infty$ shows that $\|h_{n}-h\|_p\rightarrow0$
as $n\rightarrow\infty$.
\end{proof}

It is worth mentioning that when $0<p<\infty$, \cite[pp.73, Exercise 17]{Ru4} presents two proofs for Lemma \ref{thm2104} using Egoroff's theorem and Fatou's lemma, respectively.

Recall that a subset $\Lambda$ of $L^{1}(\mathbb{T}^\infty)$ is uniformly integrable if and only if there is a non-decreasing convex function $\varphi: \mathbb{R}\rightarrow[0,\infty)$ satisfying $\frac{\varphi(t)}{t}\rightarrow\infty$ as $t\rightarrow+\infty$, called strongly convex function, such that $\{\varphi\circ|f|\}_{f\in\Lambda}$ is bounded in $L^{1}(\mathbb{T}^\infty)$, see \cite[Theorem 4.5.9]{Bog}.

We now present the proof of Proposition \ref{thm2103}.

\vskip2mm

\noindent\textbf{Proof of Proposition \ref{thm2103}.} $(1)\Rightarrow(2)$: When $F\in N_{*}(\mathbb{D}_{1}^\infty)$, the family of functions $\{\log(1+|F_{[r]}|)\}_{0<r<1}$ is uniformly integrable over $\mathbb{T}^{\infty}$. Applying Lebesgue-Vitali's theorem \cite[Theorem 4.5.4]{Bog}, we conclude that $\|F_{[r]}\|_0\rightarrow\|F^{*}\|_0$ as $r\rightarrow1$.

$(2)\Rightarrow(3)$: It immediately follows from Lemma \ref{thm2104}.

$(3)\Rightarrow(4)$: By the inequality
\begin{equation}\label{equ2103G}
|\log^{+}x-\log^{+}y|\leq\log(1+|x-y|),\quad x,y\geq0,
\end{equation}
we have
$$\left|\int_{\mathbb{T}^\infty}\log^{+}|F_{[r]}|dm_\infty-\int_{\mathbb{T}^\infty}\log^{+}|F^{*}|dm_\infty\right|\leq
\|F_{[r]}-F^{*}\|_0\rightarrow0\quad(r\rightarrow1),$$
which implies (4).

$(4)\Rightarrow(1)$: Let $\{r_n\}_{n=1}^\infty$ be an increasing sequence in $(0,1)$ with $r_n\rightarrow1$ as $n\rightarrow\infty$.
By (4) and Lemma \ref{thm2104}, the sequence $\{\log^{+}|F_{[r_n]}|\}_{n=1}^\infty$ converges to $\log^{+}|F^{*}|$ in $L^1(\mathbb{T}^\infty)$.
Then it follows from Lebesgue-Vitali's theorem that $\{\log^{+}|F_{[r_n]}|\}_{n=1}^\infty$ is uniformly integrable over $\mathbb{T}^{\infty}$.
Therefore, there is a strongly convex function $\varphi$ such that $\{\varphi(\log^{+}|F_{[r_n]}|)\}_{n=1}^\infty$ is bounded in $L^{1}(\mathbb{T}^{\infty})$.
By the $\infty-$subharmonicity of $\varphi(\log^{+}|F|)$, applying Proposition \ref{thm2101} shows that $\{\varphi(\log^{+}|F_{[r]}|)\}_{0<r<1}$ is also bounded in $L^{1}(\mathbb{T}^{\infty})$. Therefore, the family $\{\log^{+}|F_{[r]}|\}_{0<r<1}$ is uniformly integrable, which implies $F\in N_{*}(\mathbb{D}_{1}^\infty)$. $\hfill\square$

\vskip2mm

The following Proposition immediately follows from Proposition \ref{thm2103}.

\begin{prop}
$N_{*}(\mathbb{D}_{1}^\infty)$ is a closed subclass of $N(\mathbb{D}_{1}^\infty)$.
\end{prop}

\begin{proof}
Assume that $\{G_n\}_{n=1}^\infty$ is a sequence in $N_{*}(\mathbb{D}_{1}^\infty)$, and $\|G_n-G\|_0\rightarrow0$ for some $G\in N(\mathbb{D}_{1}^\infty)$ as $n\rightarrow\infty$.
To prove $G\in N_{*}(\mathbb{D}_{1}^\infty)$, let $0<r<1$ and $n\in\mathbb{N}$, then
$$\begin{aligned}
\|G_{[r]}-G^{*}\|_0&\leq\|G_{[r]}-(G_{n})_{[r]}\|_0+\|(G_{n})_{[r]}-G_{n}^{*}\|_0+\|G_{n}^{*}-G^{*}\|_0\\
&\leq2\|G_n-G\|_0+\|(G_{n})_{[r]}-G_{n}^{*}\|_0.
\end{aligned}$$
Since $G_n\in N_{*}(\mathbb{D}_{1}^\infty)$, combining the above inequality with Proposition \ref{thm2103} shows that
$$\limsup_{r\rightarrow1}\|G_{[r]}-G^{*}\|_0\leq2\|G_n-G\|_0.$$
Letting $n\rightarrow\infty$, we see that $\|G_{[r]}-G^{*}\|_0\rightarrow0$ as $r\rightarrow1$. Again by Proposition \ref{thm2103}, $G\in N_{*}(\mathbb{D}_{1}^\infty)$.
\end{proof}

Applying the inequality
$$\log(1+xy)\leq\log(1+x)+\log(1+y),\quad x,y\geq0$$
yields that the Nevanlinna class $N(\mathbb{D}_{1}^\infty)$ is an algebra, and the Smirnov class $N_{*}(\mathbb{D}_{1}^\infty)$ is its subalgebra.
Unfortunately, as shown in \cite{Dav, SS} for one-variable case, the Nevanlinna class $N(\mathbb{D}_{1}^\infty)$ is not a topological linear space since the scalar multiplication is not continuous.
However, the following proposition shows that the Smirnov class $N_{*}(\mathbb{D}_{1}^\infty)$ is in fact a topological algebra.

\begin{prop}\label{thm2105Z}
If $\{F_n\}_{n=1}^\infty$ and $\{G_n\}_{n=1}^\infty$ are two sequences in $N_{*}(\mathbb{D}_{1}^\infty)$ that converge to $F,G$ in $N_{*}(\mathbb{D}_{1}^\infty)$, respectively, then $\{F_nG_n\}_{n=1}^\infty$ converges to $FG$ in $N_{*}(\mathbb{D}_{1}^\infty)$.
\end{prop}

\begin{proof}
From Proposition \ref{thm2103}, we have $\|F_{n}^{*}-F^{*}\|_0\rightarrow0$, and $\|G_{n}^{*}-G^{*}\|_0\rightarrow0$ as $n\rightarrow\infty$.
Passing to subsequences, assume that $\{F_{n}^{*}\}_{n=1}^\infty$ and $\{G_{n}^{*}\}_{n=1}^\infty$ converge almost everywhere to $F^{*}$ and $G^{*}$ on $\mathbb{T}^\infty$, respectively.
Write
$$f_n=\log(1+|F_{n}^{*}G_{n}^{*}-F^{*}G^{*}|)$$
and
$$g_n=\log(1+|F_{n}^{*}|)+\log(1+|G_{n}^{*}-G^{*}|)+\log(1+|F_{n}^{*}-F^{*}|)+\log(1+|G^{*}|).$$
It is easy to verify that for every $n\in\mathbb{N}$, $0\leq f_n\leq g_n$ on $\mathbb{T}^\infty$.
Applying Lemma \ref{thm2103B} to sequences $\{f_n\}_{n=1}^\infty$ and $\{g_n\}_{n=1}^\infty$ shows that $\|F_{n}^{*}G_{n}^{*}-F^{*}G^{*}\|_0\rightarrow0$ as $n\rightarrow\infty$.
Then by Proposition \ref{thm2103}, $\|F_{n}G_{n}-FG\|_0\rightarrow0$ as $n\rightarrow\infty$, which completes the proof.
\end{proof}

By Proposition \ref{thmB2103}, for every $F\in N(\mathbb{D}_{1}^\infty)$, $\log|F|$ has an $\infty$-harmonic majorant in $h^{1}(\mathbb{D}_{1}^\infty)$.
The following proposition shows that if, furthermore, $F\in N_{*}(\mathbb{D}_{1}^\infty)$, the $\infty$-harmonic majorant can be taken as $\mathrm{P}[\log|F^{*}|dm_\infty]$.
The finite-variable version of this proposition is presented in \cite{Ru2}.

\begin{prop}\label{thm2108}
Suppose that $F\in N(\mathbb{D}_{1}^\infty)$. Then $F\in N_{*}(\mathbb{D}_{1}^\infty)$ if and only if for every $\zeta\in \mathbb{D}_{1}^\infty$,
\begin{equation}\label{equ2108B}
\log|F(\zeta)|\leq\int_{\mathbb{T}^\infty}\mathbf{P}_\zeta\log|F^{*}|dm_\infty.
\end{equation}
\end{prop}

\begin{proof}
We first assume that $F\in N_{*}(\mathbb{D}_{1}^\infty)$. Then for each $0<r<1$, $F_{[r]}\in A(\mathbb{D}^\infty)$, and (\ref{equB21011}) implies that for $\zeta\in\mathbb{D}_{1}^\infty$,
\begin{equation}\label{equ2109}
\log|F_{[r]}(\zeta)|\leq\int_{\mathbb{T}^\infty}\mathbf{P}_\zeta\log|F_{[r]}|dm_\infty.
\end{equation}
Since $F\in N_{*}(\mathbb{D}_{1}^\infty)$ and $\mathbf{P}_\zeta\in C(\mathbb{T}^\infty)$, the family $\{\mathbf{P}_\zeta\log^{+}|F_{[r]}|\}_{0<r<1}$ is uniformly integrable. By Lebesgue-Vitali's theorem \cite[Theorem 4.5.4]{Bog},
\begin{equation}\label{equ2110}
\lim_{r\rightarrow1}\int_{\mathbb{T}^\infty}\mathbf{P}_\zeta\log^{+}|F_{[r]}|dm_\infty=
\int_{\mathbb{T}^\infty}\mathbf{P}_\zeta\log^{+}|F^{*}|dm_\infty.
\end{equation}
On the other hand, it follows from Fatou's lemma that
\begin{equation}\label{equ2111}
\liminf\limits_{r\rightarrow1}\int_{\mathbb{T}^\infty}\mathbf{P}_\zeta\log^{-}|F_{[r]}|dm_\infty
\geq\int_{\mathbb{T}^\infty}\mathbf{P}_\zeta\log^{-}|F^{*}|dm_\infty,
\end{equation}
where $\log^{-}x=-\min\{0,\log x\}$ for $x>0$. By (\ref{equ2110}) and (\ref{equ2111}) we see that
$$\limsup\limits_{r\rightarrow1}\int_{\mathbb{T}^\infty}\mathbf{P}_\zeta\log|F_{[r]}|dm_\infty\leq
\int_{\mathbb{T}^\infty}\mathbf{P}_\zeta\log|F^{*}|dm_\infty.$$
Letting $r\rightarrow1$ in (\ref{equ2109}) shows that
$$\log|F(\zeta)|\leq\limsup\limits_{r\rightarrow1}\int_{\mathbb{T}^\infty}\mathbf{P}_\zeta\log|F_{[r]}|dm_\infty\leq
\int_{\mathbb{T}^\infty}\mathbf{P}_\zeta\log|F^{*}|dm_\infty.$$

Conversely, suppose that $F\neq0$ and (\ref{equ2108B}) holds for every $\zeta\in\mathbb{D}_{1}^\infty$, then for each $0<r<1$ and $w\in\mathbb{T}^\infty$,
$$\log|F_{[r]}(w)|\leq\int_{\mathbb{T}^\infty}\mathbf{P}_{r\star w}\log|F^{*}|dm_\infty,$$
where $r\star w\in\mathbb{D}_{1}^\infty$ by definition (\ref{equ2102C}). Since $\log|F^{*}|\in L^1(\mathbb{T}^\infty)$,  there is a strongly convex function $\varphi$ such that $\varphi(\log|F^{*}|)\in L^1(\mathbb{T}^\infty)$.
As done in finite-variable cases by Rudin \cite{Ru2}, applying the convexity of $\varphi$ and Fubini's theorem, we obtain that for $0<r<1$,
\begin{equation}\label{equ2111A}
\begin{aligned}
\int_{\mathbb{T}^\infty}\varphi(\log|F_{[r]}|)dm_\infty&\leq\int_{\mathbb{T}^\infty}\varphi\left(\int_{\mathbb{T}^\infty}\mathbf{P}_{r\star w}(\xi)\log|F^{*}(\xi)|dm_\infty(\xi)\right)dm_\infty(w)\\
&\leq\int_{\mathbb{T}^\infty}\int_{\mathbb{T}^\infty}\mathbf{P}_{r\star w}(\xi)\varphi(\log|F^{*}(\xi)|)dm_\infty(\xi)dm_\infty(w)\\
&=\int_{\mathbb{T}^\infty}\left(\int_{\mathbb{T}^\infty}\mathbf{P}_{r\star w}(\xi)dm_\infty(w)\right)\varphi(\log|F^{*}(\xi)|)dm_\infty(\xi)\\
&=\int_{\mathbb{T}^\infty}\varphi(\log|F^{*}|)dm_\infty.
\end{aligned}
\end{equation}
Let $E_r$ be the set of points $w\in\mathbb{T}^\infty$ such that $|F_{[r]}(w)|<1$, and $E_r^c$ the complement of $E_r$ with respect to $\mathbb{T}^\infty$. Then (\ref{equ2111A}) implies that
\begin{equation}\label{equ2111B}
\begin{aligned}
\int_{\mathbb{T}^\infty}\varphi(\log^{+}|F_{[r]}|)dm_\infty&=\int_{E_r}\varphi(0)dm_\infty+\int_{E_r^c}\varphi(\log|F_{[r]}|)dm_\infty\\
&\leq\varphi(0)+\int_{\mathbb{T}^\infty}\varphi(\log|F_{[r]}|)dm_\infty\\
&\leq\varphi(0)+\int_{\mathbb{T}^\infty}\varphi(\log|F^{*}|)dm_\infty.
\end{aligned}
\end{equation}
Therefore, the family $\{\log^{+}|F_{[r]}|\}_{0<r<1}$ is uniformly integrable over $\mathbb{T}^\infty$, which implies $F\in N_{*}(\mathbb{D}_{1}^\infty)$.
\end{proof}

Motivated by Proposition \ref{thm2108}, we ask when the inequality in this proposition attains an equality at some point, and such a nonzero function is called an outer function.

\begin{prop}\label{thm501}
If $F$ is a nonzero function in $N_{*}(\mathbb{D}_1^\infty)$, then the following statements are equivalent:

(1) $F$ is outer.

(2) For every $\zeta\in \mathbb{D}_1^\infty$, $\log|F(\zeta)|=\int_{\mathbb{T}^\infty}\mathbf{P}_\zeta\log|F^{*}|dm_\infty$.

(3) $\log|F(0)|=\int_{\mathbb{T}^\infty}\log|F^{*}|dm_\infty$.
\end{prop}

To prove Proposition \ref{thm501}, the maximum principle of $\infty$-subharmonic functions is needed.

\begin{lem}[The maximum principle]\label{thmB2104}
Let $u$ be an $\infty$-subharmonic function on $\mathbb{D}_{1}^\infty$, and $a\in\mathbb{R}$.
If $u\leq a$ on $\mathbb{D}_{1}^\infty$ and $u(\eta)=a$ for some $\eta\in\mathbb{D}_1^\infty$, then $u\equiv a$ on $\mathbb{D}_{1}^\infty$.
\end{lem}

\begin{proof}
For every $n\in\mathbb{N}$, set
$$u_n(z)=u(z_1,\ldots,z_n,\eta_{n+1},\eta_{n+2},\ldots),\quad z=(z_1,\ldots,z_n)\in\mathbb{D}^n.$$
Then $u_n\leq a$ is subharmonic in variables $z_1,\ldots,z_n$ separately, and $u_n(\eta_1\ldots,\eta_n)=a$.
By using the subharmonicity successively in each variable, we see that $u_n\equiv a$ on $\mathbb{D}^n$.
Now it follows from the upper semicontinuity of $u$ that for every $\zeta\in\mathbb{D}_{1}^\infty$,
$$u(\zeta)\geq\limsup_{n\rightarrow\infty}u(\zeta_1,\ldots,\zeta_n,\eta_{n+1},\eta_{n+2},\ldots)
=\limsup_{n\rightarrow\infty}u_n(\zeta_1,\ldots,\zeta_n)=a,$$
which leads to the desired conclusion.
\end{proof}

\noindent\textbf{Proof of Proposition \ref{thm501}.} It suffices to show that (1) implies (2). Put
$$\rho(\zeta)=\log|F(\zeta)|-\int_{\mathbb{T}^\infty}\mathbf{P}_{\zeta}\log|F^{*}|dm_\infty,\quad\zeta\in\mathbb{D}_1^\infty.$$
Since $\log|F^{*}|\in L^{1}(\mathbb{T}^\infty)$, $\rho$ is $\infty$-subharmonic on $\mathbb{D}_1^\infty$. By Proposition \ref{thm2108}, $\rho\leq0$ on $\mathbb{D}_1^\infty$.
Noticing that $\rho(\xi)=0$ for some $\xi\in\mathbb{D}_1^\infty$, it follows from Lemma \ref{thmB2104} that $\rho$ is identically zero on $\mathbb{D}_1^\infty$, as desired. $\hfill\square$

\vskip2mm

For every nonzero function $F\in N_{*}(\mathbb{D}_1^\infty)$ and $\zeta\in\mathbb{D}_{1}^\infty$, $\mathbf{P}_\zeta\log|F^{*}|\in L^{1}(\mathbb{T}^\infty)$. A combination of this fact and Proposition \ref{thm501} immediately gives the following result.

\begin{cor}\label{thm502}
If $F\in N_{*}(\mathbb{D}_1^\infty)$ is outer, then $F$ is zero-free in $\mathbb{D}_{1}^\infty$.
\end{cor}

For $0<p\leq\infty$, the proof of Proposition \ref{thm2108} implies that $H^p(\mathbb{D}_{2}^\infty)\subset N_{*}(\mathbb{D}_1^\infty)$.
Therefore, when $F\in H^p(\mathbb{D}_{2}^\infty)$ is outer, $F$ is zero-free in $\mathbb{D}_{1}^\infty$.
In fact, it is also zero-free in $\mathbb{D}_{2}^\infty$.

\begin{cor}\label{thm503}
Let $0<p\leq\infty$. If $F\in H^p(\mathbb{D}_{2}^\infty)$ is outer, then $F$ is zero-free in $\mathbb{D}_{2}^\infty$.
\end{cor}

\begin{proof}
Assume that there exists $\zeta\in\mathbb{D}_{2}^\infty$ such that $F(\zeta)=0$. Choose $\eta\in\mathbb{D}$ and an integer $N$ such that
$$(\eta^{-1}\zeta_{1},\ldots,\eta^{-1}\zeta_{N},2\zeta_{N+1},2\zeta_{N+2},\ldots)\in\mathbb{D}_{2}^\infty.$$
For each $n\in\mathbb{N}$, set
$$\phi_{n}(z)=F(z_{1}\eta^{-1}\zeta_{1},\ldots,z_{1}\eta^{-1}\zeta_{N},2z_{2}\zeta_{N+1},\ldots,2z_{2}\zeta_{N+n},0,\ldots),\quad z=(z_1,z_2)\in\mathbb{D}^{2},$$
and
$$\phi(z)=F(z_{1}\eta^{-1}\zeta_{1},\ldots,z_{1}\eta^{-1}\zeta_{N},2z_{2}\zeta_{N+1},2z_{2}\zeta_{N+2},\ldots),\quad z=(z_1,z_2)\in\mathbb{D}^{2}.$$
Then $\phi_{n}$ is holomorphic on $\mathbb{D}^{2}$, and $\{\phi_{n}\}_{n=1}^{\infty}$ converges uniformly to $\phi$ on each compact subset of $\mathbb{D}^{2}$. Therefore, $\phi$ is holomorphic on $\mathbb{D}^{2}$. Since $F$ is zero-free in $\mathbb{D}_{1}^\infty$, $\phi_{n}$ is zero-free in $\mathbb{D}^{2}$. Noticing that $\phi(\eta,\frac{1}{2})=0$, Hurwitz's theorem \cite[pp.310, Exercise 3]{Kra} implies that $\phi$ is identically zero on $\mathbb{D}^2$, and hence
$$F(\zeta_{1},\ldots,\zeta_{N},0,\ldots)=\phi(\eta,0)=0,$$
a contradiction to that $F$ is zero-free in $\mathbb{D}_{1}^\infty$.
\end{proof}

The following corollary will be used in Section 3.

\begin{cor}\label{thm503B}
If $F\in N_{*}(\mathbb{D}_1^\infty)$ is outer, then $\frac{1}{F}\in N_{*}(\mathbb{D}_1^\infty)$ is outer.
\end{cor}

\begin{proof}
Since $F$ is outer, it follows from Corollary \ref{thm502} that $F$ is zero-free in $\mathbb{D}_{1}^\infty$.
Hence by Proposition \ref{thmB2101}, $\frac{1}{F}\in N(\mathbb{D}_1^\infty)$.
On the other hand, for every $\zeta\in\mathbb{D}_{1}^\infty$,
$$\log|F(\zeta)|=\int_{\mathbb{T}^\infty}\mathbf{P}_{\zeta}\log|F^{*}|dm_\infty,$$
and hence
$$\log\left|\frac{1}{F(\zeta)}\right|=\int_{\mathbb{T}^\infty}\mathbf{P}_{\zeta}\log\left|\frac{1}{F^{*}}\right|dm_\infty.$$
Combining this equality and Proposition \ref{thm2108} gives $\frac{1}{F}\in N_{*}(\mathbb{D}_1^\infty)$ is outer.
\end{proof}

In what follows we consider cyclic vectors.
A closed subspace $S\subset N_{*}(\mathbb{D}_1^\infty)$ is said to be invariant, if for every $F\in S$ and $q\in\mathcal{P}_\infty$, $qF\in S$.
We say that a function $F\in N_{*}(\mathbb{D}_1^\infty)$ is cyclic, if the invariant subspace generated by $F$ is exactly $N_{*}(\mathbb{D}_1^\infty)$.
For $0<p<\infty$, cyclic vectors in the Hardy space $H^{p}(\mathbb{D}_{2}^\infty)$ can be defined similarly.

The following theorem gives a quantitative description of cyclic vectors in $N_{*}(\mathbb{D}_1^\infty)$. Although its proof is similar as in \cite[Theorem 4.4.6]{Ru2}, we present it here for completeness.

\begin{thm}\label{thm504}
Each cyclic vector in $N_{*}(\mathbb{D}_1^\infty)$ is outer.
\end{thm}

\begin{proof}
For a nonzero function $F\in N_{*}(\mathbb{D}_1^\infty)$, set
$$\Gamma F=\int_{\mathbb{T}^\infty}\log|F^{*}|dm_\infty-\log|F(0)|.$$
Then Proposition \ref{thm501} yields that $F$ is outer if and only if $\Gamma F=0$.
We claim that $\Gamma$ is upper semicontinuous on $N_{*}(\mathbb{D}_1^\infty)\backslash\{0\}$.
Indeed, let $\{G_n\}_{n=1}^\infty$ be a sequence in $N_{*}(\mathbb{D}_1^\infty)\backslash\{0\}$ and $G\in N_{*}(\mathbb{D}_1^\infty)\backslash\{0\}$ such that for all $n\in\mathbb{N}$, $\Gamma G_n\geq c\geq0$, and $\|G_n-G\|_0\rightarrow0$ as $n\rightarrow\infty$.
Hence by (\ref{equ2103G}) and Proposition \ref{thm2103},
$$\left|\int_{\mathbb{T}^\infty}\log^{+}|G_n^{*}|dm_\infty-\int_{\mathbb{T}^\infty}\log^{+}|G^{*}|dm_\infty\right|\leq\|G_n^{*}-G^{*}\|_0=\|G_n-G\|_0,$$
and thus
\begin{equation}\label{equ503}
\lim_{n\rightarrow\infty}\int_{\mathbb{T}^\infty}\log^{+}|G_n^{*}|dm_\infty=\int_{\mathbb{T}^\infty}\log^{+}|G^{*}|dm_\infty.
\end{equation}
On the other hand, by the convergence of $\{G_n^*\}_{n=1}^\infty$ in $L^{0}(\mathbb{T}^{\infty})$,
there exists a subsequence $\{G_{n_k}^{*}\}_{k=1}^\infty$ converging to $G^*$ almost everywhere on $\mathbb{T}^\infty$.
Then it follows from Fatou's lemma that
\begin{equation}\label{equ504}
\int_{\mathbb{T}^\infty}\log^{-}|G^{*}|dm_\infty\leq\liminf_{k\rightarrow\infty}\int_{\mathbb{T}^\infty}\log^{-}|G_{n_k}^*|dm_\infty.
\end{equation}
A combination of (\ref{equ503}) and (\ref{equ504}) shows that
$$\int_{\mathbb{T}^\infty}\log|G^*|dm_\infty\geq\limsup_{k\rightarrow\infty}\int_{\mathbb{T}^\infty}\log|G_{n_k}^*|dm_\infty.$$
Noticing that $\|G_n-G\|_0\rightarrow0$ as $n\rightarrow\infty$,
Lemma \ref{thm2105B} gives $G_n(0)\rightarrow G(0)$ as $n\rightarrow\infty$. Therefore, $\Gamma G\geq\limsup_{k\rightarrow\infty}\Gamma G_{n_{k}}\geq c$, and hence the claim holds.

Now we assume that $F$ is cyclic in $N_{*}(\mathbb{D}_1^\infty)$, then there is a sequence of polynomials $\{q_n\}_{n=1}^\infty$ in $\mathcal{P}_\infty$ such that $\|q_n F-1\|_0\rightarrow0\;(n\rightarrow\infty)$. Hence
$$0=\Gamma 1\geq\limsup_{n\rightarrow\infty}\Gamma(q_n F)\geq\limsup_{n\rightarrow\infty}\Gamma q_n+\Gamma F\geq \Gamma F,$$
forcing $\Gamma F=0$, and thus $F$ is outer.
\end{proof}

We mention that for every $0<p<\infty$, there is a constant $C_p>0$ such that
\begin{equation}\label{equ2200}
\|f\|_0\leq C_p\|f\|_p^{\min\{p,1\}},\quad f\in L^{p}(\mathbb{T}^\infty).
\end{equation}
Hence for every $F\in H^p(\mathbb{D}_{2}^\infty)\subset N_*(\mathbb{D}_{1}^\infty)$,
$$\|F\|_0=\sup_{0<r<1}\|F_{[r]}\|_0\leq C_p\sup_{0<r<1}\|F_{[r]}\|_p=C_p\|F\|_p.$$
A combination of this fact and Theorem \ref{thm504} gives the following corollary.

\begin{cor}\label{thm505}
If $0<p<\infty$, then each cyclic vector in $H^{p}(\mathbb{D}_{2}^\infty)$ is outer.
\end{cor}

\subsection{Correspondence between $N_{*}(\mathbb{T}^\infty)$ and $N_{*}(\mathbb{D}_1^\infty)$}

The purpose of this subsection is to show that there is a canonical isometric isomorphism between $N_{*}(\mathbb{T}^\infty)$ and $N_{*}(\mathbb{D}_1^\infty)$. We refer readers to \cite{AOS2, DG3} for cases of the Hardy spaces in infinitely many variables.

Recall that the Smirnov class $N_{*}(\mathbb{T}^\infty)$ over the infinite torus is defined to be the closure of $\mathcal{P}_\infty$ in $L^{0}(\mathbb{T}^\infty)$.
A similar argument as in Proposition \ref{thm2105Z} shows that $N_{*}(\mathbb{T}^\infty)$ is a topological algebra.
By (\ref{equ2200}), for every $0<p<\infty$, the space $H^{p}(\mathbb{T}^\infty)$ is contained in $N_{*}(\mathbb{T}^\infty)$. For each point $\zeta\in\mathbb{D}_1^\infty$ and $q\in\mathcal{P}_\infty$, it follows from Lemma \ref{thm2105B} that
$$\log(1+|q(\zeta)|)\leq\|\mathbf{P}_\zeta\|_\infty\int_{\mathbb{T}^\infty}\log(1+|q|)dm_\infty.$$
This means that the evaluation functional $E_\zeta: q\mapsto q(\zeta)$ is continuous on the  dense subspace $\mathcal{P}_\infty$ of $N_{*}(\mathbb{T}^\infty)$, and hence $E_\zeta$ can be continuously extended to the whole space $N_{*}(\mathbb{T}^\infty)$, still denoted by $E_\zeta$.
For every $f\in N_{*}(\mathbb{T}^\infty)$ and $\zeta\in\mathbb{D}_1^\infty$, set $\tilde{f}(\zeta)=E_\zeta f$. Then $\tilde{f}$ defines a function on $\mathbb{D}_{1}^\infty$, and
\begin{equation}\label{equ2201A}
\log(1+|\tilde{f}(\zeta)|)\leq\|\mathbf{P}_\zeta\|_\infty\int_{\mathbb{T}^\infty}\log(1+|f|)dm_\infty.
\end{equation}
When $q\in\mathcal{P}_\infty$, $\tilde{q}=q$ on $\mathbb{D}_1^\infty$, so we will no longer distinguish between $q$ and $\tilde{q}$.

For $0<r<1$ and $M>0$, the domain $V_{r,M}$ in $\ell^1$ is defined by (\ref{equ2103A}). The following lemma is obvious by using (\ref{equ2201A}). It will be used not only in the proof of Propostion \ref{thm2110}, but also in Section 3.

\begin{lem}\label{thm2109}
Suppose that $\{f_n\}_{n=1}^\infty$ is a sequence in $N_{*}(\mathbb{T}^\infty)$. If $\|f_n-f\|_0\rightarrow0$ for some $f\in N_{*}(\mathbb{T}^\infty)$ as $n\rightarrow\infty$, then $\{\tilde{f}_n\}_{n=1}^\infty$ converges to $\tilde{f}$ uniformly on each $V_{r,M}$.
\end{lem}

We now establish the correspondence between $N_{*}(\mathbb{T}^\infty)$ and $N_{*}(\mathbb{D}_1^\infty)$.

\begin{prop}\label{thm2110}
If $f\in N_{*}(\mathbb{T}^\infty)$, then $\tilde{f}\in N_{*}(\mathbb{D}_1^\infty)$.
\end{prop}

\begin{proof}
We first show  that $\tilde{f}$ is holomorphic on $\mathbb{D}_1^\infty$. Choose a sequence of polynomials $\{q_n\}_{n=1}^\infty$ such that $\|q_n-f\|_0\rightarrow0$ as $n\rightarrow\infty$. By Lemma \ref{thm2109}, $\{q_n\}_{n=1}^\infty$ converges uniformly to $\tilde{f}$ on each $V_{r,M}$, and thus $\tilde{f}$ is holomorphic on $\mathbb{D}_1^\infty$. It remains to show $\{\log(1+|\tilde{f}_{[r]}|)\}_{0<r<1}$ forms a uniformly integrable family. Since $\|q_n-f\|_0\rightarrow0$ as $n\rightarrow\infty$, applying Lebesgue-Vitali's theorem \cite[Theorem 4.5.4]{Bog} shows that $\{\log(1+|q_n|)\}_{n=1}^\infty$ is uniformly integrable. Therefore, there exists a strongly convex function $\varphi$ such that
$$\sup_{n\in\mathbb{N}}\int_{\mathbb{T}^\infty}\varphi(\log(1+|q_n|))dm_\infty<\infty.$$
By Fatou's lemma, we have
\begin{equation}\label{equ2201}
\int_{\mathbb{T}^\infty}\varphi(\log(1+|\tilde{f}_{[r]}|))dm_\infty\leq
\liminf_{n\rightarrow\infty}\int_{\mathbb{T}^\infty}\varphi(\log(1+|(q_n)_{[r]}|))dm_\infty.
\end{equation}
For every $n\in\mathbb{N}$, since $\varphi(\log(1+|q_n|))$ is continuous on $\overline{\mathbb{D}}^\infty$, and $\infty$-subharmonic on $\mathbb{D}_1^\infty$, applying Corollary \ref{thm2101B} gives that for every $0<r<1$,
$$\int_{\mathbb{T}^\infty}\varphi(\log(1+|(q_n)_{[r]}|))dm_\infty\leq\int_{\mathbb{T}^\infty}\varphi(\log(1+|q_n|))dm_\infty.$$
Taking this inequality back into (\ref{equ2201}) yields
$$\sup_{0<r<1}\int_{\mathbb{T}^\infty}\varphi(\log(1+|\tilde{f}_{[r]}|))dm_\infty\leq
\sup_{n\in\mathbb{N}}\int_{\mathbb{T}^\infty}\varphi(\log(1+|q_n|))dm_\infty<\infty,$$
which ensures the uniform integrability of $\{\log(1+|\tilde{f}_{[r]}|)\}_{0<r<1}$, as desired.
\end{proof}

Recall that for every $F\in N_{*}(\mathbb{D}_1^\infty)$, $F$'s ``boundary-valued function"
$$F^{*}(w)=\lim_{r\rightarrow1}F_{[r]}(w)$$
exists for almost every $w\in\mathbb{T}^{\infty}$. In fact, this ``boundary-valued function" belongs to the Smirnov class $N_{*}(\mathbb{T}^\infty)$ over the infinite torus.

\begin{prop}\label{thm2111}
If $F\in N_{*}(\mathbb{D}_1^\infty)$, then $F^{*}\in N_{*}(\mathbb{T}^\infty)$ and $\|F\|_0=\|F^{*}\|_0$.
\end{prop}

\begin{proof}
The equality $\|F\|_0=\|F^{*}\|_0$ immediately follows from Proposition \ref{thm2103}. Given $\varepsilon>0$, again by Proposition \ref{thm2103}, there exists $0<r<1$ such that $\|F_{[r]}-F^{*}\|_0<\varepsilon$. Since $F_{[r]}\in A(\mathbb{T}^\infty)$, we can find a polynomial $q\in\mathcal{P}_\infty$ satisfying $\|q-F_{[r]}\|_\infty<e^{\varepsilon}-1$, and hence $\|q-F_{[r]}\|_0<\varepsilon$. Therefore,
$$\|q-F^{*}\|_0\leq\|q-F_{[r]}\|_0+\|F_{[r]}-F^{*}\|_0<2\varepsilon,$$
which implies $F^{*}\in N_{*}(\mathbb{T}^\infty)$.
\end{proof}

By Proposition \ref{thm2111}, the algebra homomorphism defined by
$$\Lambda: N_{*}(\mathbb{D}_1^\infty)\rightarrow N_{*}(\mathbb{T}^\infty),\quad F\mapsto F^{*},$$
is a linear isometry. The main theorem of this subsection is stated as follows.

\begin{thm}[Generalized Fatou's theorem]\label{thm2112}
The map $\Lambda: N_{*}(\mathbb{D}_1^\infty)\rightarrow N_{*}(\mathbb{T}^\infty)$ is an isometric algebra isomorphism, and its inverse map $\Lambda^{-1}: N_{*}(\mathbb{T}^\infty)\rightarrow N_{*}(\mathbb{D}_1^\infty)$ is given by $f\mapsto\tilde{f}$.
\end{thm}

\begin{proof}
It suffices to show that for every $f\in N_{*}(\mathbb{T}^\infty)$, $\Lambda\tilde{f}=f$. Choosing a sequence of polynomials $\{q_n\}_{n=1}^\infty$ satisfying
\begin{equation}\label{equ2202}
\|q_n-f\|_0\rightarrow0\quad(n\rightarrow\infty)
\end{equation}
and applying Lemma \ref{thm2109} lead that $\{q_n\}_{n=1}^\infty$ converges to $\tilde{f}$ pointwise on $\mathbb{D}_{1}^\infty$. On the other hand, (\ref{equ2202}) also implies that $\{q_n\}_{n=1}^\infty$ is a Cauchy sequence in $N_{*}(\mathbb{D}_1^\infty)$, and hence $\|q_n-G\|_0\rightarrow0$ for some $G\in N_{*}(\mathbb{D}_1^\infty)$ as $n\rightarrow\infty$. By using Lemma \ref{thm2105B}, $\{q_n\}_{n=1}^\infty$ converges to $G$ pointwise on $\mathbb{D}_{1}^\infty$, therefore $G=\tilde{f}$. We conclude from Proposition \ref{thm2111} that
$$\|q_n-\Lambda\tilde{f}\|_0=\|q_n-\Lambda G\|_0=\|q_n-G\|_0\rightarrow0\quad(n\rightarrow\infty).$$
From this and (\ref{equ2202}), $\Lambda\tilde{f}=f$ , which yields the desired conclusion.
\end{proof}

For $0<p<\infty$, every function in $H^{p}(\mathbb{T}^\infty)$, by evaluation functional, can be extended to a holomorphic function on $\mathbb{D}_{2}^\infty$. Comparison of this fact and Theorem \ref{thm2112} leads us to ask the following question.

\begin{Qes}
Does there exist a function in $N_{*}(\mathbb{T}^\infty)$ which can not be extended to a function holomorphic on $\mathbb{D}_2^\infty$?
\end{Qes}

As an application of the generalized Fatou's theorem,
we obtain Corollary \ref{thm2113},
which will be used in Section 3.

\begin{cor}\label{thm2113}
If $0<p\leq\infty$, then $N_{*}(\mathbb{T}^\infty)\cap L^p(\mathbb{T}^\infty)=H^p(\mathbb{T}^\infty)$.
\end{cor}

\begin{proof}
Clearly, $H^p(\mathbb{T}^\infty)\subset N_{*}(\mathbb{T}^\infty)\cap L^p(\mathbb{T}^\infty)$.
Conversely, for every $f\in N_{*}(\mathbb{T}^\infty)\cap L^p(\mathbb{T}^\infty)$, we will show that $f\in H^p(\mathbb{T}^\infty)$.
Write $F=\Lambda^{-1}f$, then $F\in N_{*}(\mathbb{D}_{1}^\infty)$, and $F^{*}=f$.
By Proposition \ref{thm2108}, for each $\zeta\in\mathbb{D}_{1}^\infty$,
\begin{equation}\label{equ2203}
\log|F(\zeta)|\leq\int_{\mathbb{T}^\infty}\mathbf{P}_\zeta\log|F^{*}|dm_\infty.
\end{equation}

When $0<p<\infty$, by (\ref{equ2203}), taking $\varphi(t)=\exp(pt)$ and applying the reasoning as in (\ref{equ2111A}) show that
$$\lim_{r\rightarrow1}\int_{\mathbb{T}^\infty}|F_{[r]}|^p dm_\infty\leq\int_{\mathbb{T}^\infty}|F^{*}|^p dm_\infty.$$
On the other hand, we infer from Fatou's lemma that
$$\int_{\mathbb{T}^\infty}|F^{*}|^p dm_\infty\leq\lim_{r\rightarrow1}\int_{\mathbb{T}^\infty}|F_{[r]}|^p dm_\infty,$$
and hence $\|F^{*}\|_p=\lim_{r\rightarrow1}\|F_{[r]}\|_p$.
Then by Lemma \ref{thm2104}, $\|F_{[r]}-F^{*}\|_{p}\rightarrow0$ as $r\rightarrow1$.
Noticing that for $0<r<1$, $F_{[r]}\in A(\mathbb{T}^\infty)$, we see $f=F^{*}\in H^p(\mathbb{T}^\infty)$.

When $p=\infty$, by (\ref{equ2203}) and Jensen's inequality, for every $\zeta\in\mathbb{D}_{1}^\infty$,
$$|F(\zeta)|\leq\exp\left(\int_{\mathbb{T}^\infty}\mathbf{P}_\zeta\log|F^{*}|dm_\infty\right)\leq
\int_{\mathbb{T}^\infty}\mathbf{P}_\zeta|F^{*}|dm_\infty\leq\|F^{*}\|_\infty<\infty,$$
which implies that $F$ is a bounded holomorphic function on $\mathbb{D}_1^\infty$.
By the argument in \cite[pp.7-8]{HLS}, $F$ can be extended to a bounded holomorphic function on $\mathbb{D}_{2}^\infty$, still denoted by $F$.
Then $F\in H^\infty(\mathbb{D}_{2}^\infty)$, and thus $f=F^{*}\in H^\infty(\mathbb{T}^\infty)$.
\end{proof}

\section{Szeg\"{o}'s problem in infinitely many variables}

In this section, we will apply function theory of the Smirnov class to discuss Szeg\"{o}'s problem in infinitely many variables.
Write $\mathcal{P}_0$ for the set of polynomials $q\in\mathcal{P}_\infty$ for which $q(0)=0$.
In what follows we assume that $1<p<\infty$, and $K\in L^1(\mathbb{T}^\infty)$ is a nonnegative function with $\log K\in L^1(\mathbb{T}^\infty)$.
Without loss of generality, suppose that $\|K\|_1=1$.
Motivated by Szeg\"{o}'s theorem mentioned in Introduction, one is naturally concerned with the following quantity:
$$S(K)=\inf_{q\in\mathcal{P}_0}\int_{\mathbb{T}^\infty} |1-q|^p K dm_{\infty}.$$
Obviously, $S(K)\leq1$.
On the other hand, it follows from Jensen's inequality and (\ref{equB21011}) that for every $q\in \mathcal{P}_0$,
$$\begin{aligned}
\int_{\mathbb{T}^\infty}|1-q|^pK dm_{\infty}&\geq\exp\left(\int_{\mathbb{T}^\infty}\log|1-q|^pdm_{\infty}+\int_{\mathbb{T}^\infty}\log K
dm_{\infty}\right)\\
&\geq\exp\left(\log|1-q(0)|^p+\int_{\mathbb{T}^\infty}\log K
dm_{\infty}\right)\\
&=\exp\left(\int_{\mathbb{T}^\infty}\log K dm_{\infty}\right).
\end{aligned}$$
Therefore, $S(K)$ falls into the closed interval $[\exp\left(\int_{\mathbb{T}^\infty}\log K dm_{\infty}\right), 1]$.
Then it is natural to ask for which $K$, $S(K)$ attains the lower bound $\exp\left( \int_{\mathbb{T}^\infty} \log K dm_{\infty}\right)$ or the upper bound $1$, and this is called Szeg\"{o}'s problem.

We first study when $S(K)$ attains the lower bound $\exp\left(\int_{\mathbb{T}^\infty}\log K dm_{\infty}\right)$.
The next theorem gives the answer.

\begin{thm}\label{thm601}
The equality
\begin{equation}\label{equ601}
S(K)=\exp\left(\int_{\mathbb{T}^\infty}\log K dm_{\infty}\right)
\end{equation}
holds if and only if $K=|h|^p$ for some cyclic vector $h\in H^p(\mathbb{T}^\infty)$.
\end{thm}

To prove Theorem \ref{thm601}, we need the following proposition which states that the weighted Hardy space $H^p(K dm_\infty)$,
the closure of $\mathcal{P}_\infty$ in $L^p(K dm_{\infty})$,
 is a subset of the Smirnov class $N_{*}(\mathbb{T}^\infty)$.

\begin{prop}\label{thm602}
$H^p(K dm_\infty)\subset N_*(\mathbb{T}^\infty)$.
\end{prop}

\begin{proof}
For every $h\in H^p(K dm_\infty)$, there is a sequence of polynomials $\{q_n\}_{n=1}^\infty$ for which
\begin{equation}\label{equ602}
\int_{\mathbb{T}^\infty}|q_n-h|^pK dm_\infty\rightarrow0\quad(n\rightarrow\infty).
\end{equation}
Then there is a subsequence $\{q_{n_{k}}\}_{k=1}^\infty$ that converges to $h$ almost everywhere on $\mathbb{T}^\infty$.
Moreover, by the inequality $\log(1+x)\leq x\;(x\geq0)$ and (\ref{equ602}), we have
$$\int_{\mathbb{T}^\infty}\log(1+|q_{n_k}-h|^pK)dm_\infty\rightarrow0\quad(k\rightarrow\infty).$$
Write
$$f_{n_k}=\log(1+|q_{n_k}-h|^p),\quad g_{n_k}=\log(1+|q_{n_k}-h|^pK)+\log(1+K)-\log K.$$
For every $k\in\mathbb{N}$, it is easy to verify that $0\leq f_{n_k}\leq g_{n_k}$ on $\mathbb{T}^\infty$.
Applying Lemma \ref{thm2103B} to sequences $\{f_{n_k}\}_{k=1}^\infty$ and $\{g_{n_k}\}_{k=1}^\infty$, we obtain
\begin{equation}\label{equ602B}
\int_{\mathbb{T}^\infty}\log(1+|q_{n_k}-h|^p)dm_\infty\rightarrow0\quad(k\rightarrow\infty).
\end{equation}
Let $E_k$ be the set of points $w\in\mathbb{T}^\infty$ such that $|q_{n_k}(w)-h(w)|\leq1$, and $E_k^c$ the complement of $E_k$ with respect to $\mathbb{T}^\infty$.
Write $\chi_{E_k}$ for the characteristic function of $E_k$.
Then for every $k\in\mathbb{N}$, $0\leq\chi_{E_k}|q_{n_k}-h|\leq1$ on $\mathbb{T}^\infty$.
It follows from Lebesgue's dominated convergence theorem that
\begin{equation}\label{equ603}
\int_{E_k}\log(1+|q_{n_k}-h|)dm_\infty=\int_{\mathbb{T}^\infty}\log(1+\chi_{E_k}|q_{n_k}-h|)dm_\infty\rightarrow0\quad(k\rightarrow\infty).
\end{equation}
On the other hand, by inequalities
$$\int_{E_k^c}\log(1+|q_{n_k}-h|)dm_\infty\leq\int_{E_k^c}\log(1+|q_{n_k}-h|^p)dm_\infty\leq\int_{\mathbb{T}^\infty}\log(1+|q_{n_k}-h|^p)dm_\infty,$$
we conclude from (\ref{equ602B}) that
$$\int_{E_k^c}\log(1+|q_{n_k}-h|)dm_\infty\rightarrow0\quad(k\rightarrow\infty).$$
Combining this with (\ref{equ603}) gives $\|q_{n_k}-h\|_0\rightarrow0$ as $k\rightarrow\infty$, and thus $h\in N_*(\mathbb{T}^\infty)$, as desired.
\end{proof}

Recall that to every $f\in N_*(\mathbb{T}^\infty)$ corresponds a holomorphic function $\tilde{f}\in N_*(\mathbb{D}_1^\infty)$ whose radial limit is $f$.
We are now ready to prove Theorem \ref{thm601}.

\vskip2mm

\noindent\textbf{Proof of Theorem \ref{thm601}.}
We first assume $K=|h|^p$ for some cyclic vector $h\in H^p(\mathbb{T}^\infty)$. To prove (\ref{equ601}), it suffices to show that
\begin{equation}\label{equ604}
S(K)\leq\exp\left(\int_{\mathbb{T}^\infty}\log K dm_\infty\right).
\end{equation}
By the cyclicity of $h$, it is not difficult to verify that $h-\tilde{h}(0)$ belongs to the closure of
$\left\{qh\in H^p(\mathbb{T}^\infty): q\in \mathcal{P}_0\right\}$
in $H^p(\mathbb{T}^\infty)$.
Therefore,
\begin{equation}\label{equ605}
S(K)=\inf_{q\in \mathcal{P}_0}\int_{\mathbb{T}^\infty}|h-qh|^p dm_{\infty}\leq
\int_{\mathbb{T}^\infty}|h-(h-\tilde{h}(0))|^p dm_{\infty}=|\tilde{h}(0)|^p.
\end{equation}
Since $h\in H^p(\mathbb{T}^\infty)$ is outer, that is,
$$|\tilde{h}(0)|^p=\exp\left(\int_{\mathbb{T}^\infty}\log|h|^pdm_\infty\right)=\exp\left(\int_{\mathbb{T}^\infty}\log K dm_\infty\right),$$
we deduce from this and (\ref{equ605}) that $K$ satisfies (\ref{equ604}).

Conversely, we assume that (\ref{equ601}) holds.
Since $H^p(K dm_\infty)$ is reflexive, there exists a function $\varphi$ belonging to the closure of $\mathcal{P}_0$ in $H^p(K dm_\infty)$ such that
\begin{equation}\label{equ605B}
S(K)=\int_{\mathbb{T}^\infty}|1-\varphi|^pK dm_{\infty}.
\end{equation}
Choose a sequence of polynomials $\{q_n\}_{n=1}^\infty$ in $\mathcal{P}_0$ satisfying
\begin{equation}\label{equ605C}
\int_{\mathbb{T}^\infty}|q_n-\varphi|^pK dm_\infty\rightarrow0\quad(n\rightarrow\infty).
\end{equation}
It follows from Proposition \ref{thm602} that $\varphi\in N_*(\mathbb{T}^\infty)$, and the proof of this proposition also implies that there is a subsequence $\{q_{n_k}\}_{k=1}^{\infty}$ for which $\|q_{n_k}-\varphi\|_0\rightarrow0$ as $k\rightarrow\infty$.
Hence by Lemma \ref{thm2109}, $\{q_{n_k}\}_{k=1}^{\infty}$ converges to $\tilde{\varphi}$ pointwise on $\mathbb{D}_1^\infty$, and thus $\tilde{\varphi}(0)=0$.
We see from Proposition \ref{thm2108} that
\begin{equation}\label{equ606}
\int_{\mathbb{T}^\infty}\log |1-\varphi|^p dm_{\infty}\geq\log |1-\tilde{\varphi}(0)|^p=0.
\end{equation}
By Jensen's inequality,
\begin{equation}\label{equ607}
\int_{\mathbb{T}^\infty}|1-\varphi|^pK dm_{\infty}\geq\exp\left(\int_{\mathbb{T}^\infty}\log|1-\varphi|^p+\log K dm_\infty\right),
\end{equation}
and the equality holds if and only if $|1-\varphi|^p K$ is a constant. Combining (\ref{equ606}) with (\ref{equ607}) shows that
$$\int_{\mathbb{T}^\infty}|1-\varphi|^pK dm_{\infty}\geq\exp\left(\int_{\mathbb{T}^\infty}\log K dm_\infty\right).$$
By (\ref{equ601}) and (\ref{equ605B}), this inequality is actually an equality, forcing (\ref{equ606}) and (\ref{equ607}) to be equalities.
This means that $1-\varphi\in N_{*}(\mathbb{T}^\infty)$ is outer, and there is a constant $C>0$ such that $|1-\varphi|^p K=C^p$.

We claim that $\frac{1}{1-\varphi}\in H^{p}(\mathbb{T}^\infty)$.
Indeed, since $1-\varphi\in N_{*}(\mathbb{T}^\infty)$ is outer, by Corollary \ref{thm503B} and the generalized Fatou's theorem, $\frac{1}{1-\varphi}\in N_{*}(\mathbb{T}^\infty)$, and it is outer.
On the other hand, since $\frac{1}{|1-\varphi|^p}=\frac{K}{C^p}$, we have $\frac{1}{1-\varphi}\in L^p(\mathbb{T}^\infty)$.
Then by Corollary \ref{thm2113}, $\frac{1}{1-\varphi}\in H^{p}(\mathbb{T}^\infty)$, and the claim holds.

Write $h=\frac{C}{1-\varphi}\in H^{p}(\mathbb{T}^\infty)$, then $K=|h|^p$.
We proceed to prove that $h$ is cyclic by (\ref{equ605C}).
Noticing that
$$|q_n-\varphi|^pK=|(1-q_n)-(1-\varphi)|^p|h|^p=|(1-q_n)h-C|^p,$$
hence by (\ref{equ605C}),
$$\int_{\mathbb{T}^\infty}|(1-q_n)h-C|^pdm_\infty=\int_{\mathbb{T}^\infty}|q_n-\varphi|^pK dm_\infty\rightarrow0\quad(n\rightarrow\infty).$$
The above reasoning shows that the constant function $C$ is in the invariant subspace of $H^p(\mathbb{T}^\infty)$ generated by $h$.
Therefore, $h$ is cyclic, and the proof is complete. $\hfill\square$

\vskip2mm

We continue to discuss when $S(K)$ attains the upper bound $1$.
Let $\mathbb{Z}_{0}^{\infty}$ be the set of all finitely supported sequences of integers.
Recall that for every $\alpha\in\mathbb{Z}_{0}^{\infty}$, the Fourier coefficient $\hat{K}(\alpha)$ of $K$ is defined to be
$$\hat{K}(\alpha)=\int_{\mathbb{T}^\infty}K\bar{e}_\alpha dm_\infty,$$
where
$$e_\alpha(w)=w_1^{\alpha_1}w_2^{\alpha_2}\cdots,\quad w\in\mathbb{T}^\infty.$$

\begin{thm}\label{thm603}
$S(K)=1$ if and only if
\begin{equation}\label{equ608B}
\hat{K}(\alpha)=\hat{K}(-\alpha)=0,\quad 0\neq\alpha\in\mathbb{N}_{0}^{\infty},
\end{equation}
where
$\mathbb{N}_{0}^{\infty}$ denotes the set of all finitely supported sequences of nonnegative integers.
\end{thm}

\begin{proof}
Let $H_0^p(Kdm_\infty)$ denote the closure of $\mathcal{P}_0$ in $L^p(K dm_\infty)$.
Applying Hahn-Banach theorem shows that the quotient space $\mathcal{Q}=L^p(K dm_\infty)/H_0^p(Kdm_\infty)$ is of the dual space $\mathcal{Q}^{*}$ as follows:
$$\mathcal{Q}^{*}=\left\{g\in L^q(K dm_\infty):\;\int_{\mathbb{T}^\infty}fgKdm_\infty=0\;\;\text{for}\;\;f\in H_0^p(Kdm_\infty)\right\},$$
where $\frac{1}{p}+\frac{1}{q}=1$.
Therefore,
\begin{equation}\label{equ609}
S(K)=\inf_{f\in H_0^p(Kdm_\infty)}\int_{\mathbb{T}^\infty}|1-f|^pK dm_{\infty}=\|\pi(1)\|_{\mathcal{Q}}^p
=\sup_{g\in \mathbf{B}\cap \mathcal{Q}^{*}}\left|\int_{\mathbb{T}^\infty}gK dm_\infty\right|^p,
\end{equation}
where $\pi$ is the quotient map of $L^p(K dm_\infty)$ onto $\mathcal{Q}$, and $\mathbf{B}$ is the closed unit ball of $L^q(K dm_\infty)$.

When (\ref{equ608B}) holds, it is easy to verify that $1\in\mathbf{B}\cap\mathcal{Q}^{*}$. Then by (\ref{equ609}),
$$S(K)=\sup_{g\in \mathbf{B}\cap\mathcal{Q}^{*}}\left|\int_{\mathbb{T}^\infty}gK dm_\infty\right|^p\geq\left|\int_{\mathbb{T}^\infty}K dm_\infty\right|^p=1,$$
forcing $S(K)=1$.

Conversely, we assume that $S(K)=1$.
Since $L^q(K dm_\infty)$ is reflexive, $\mathbf{B}\cap\mathcal{Q}^{*}$ is weakly compact. Hence by (\ref{equ609}) there exists a $\psi\in\mathbf{B}\cap\mathcal{Q}^{*}$ such that
$$1=S(K)=\left|\int_{\mathbb{T}^\infty}\psi K dm_\infty\right|^p.$$
Noticing that $\|\psi\|_{L^{q}(K dm_\infty)}\leq1$, H\"{o}lder's inequality gives that $\psi=1$, and hence $1\in\mathcal{Q}^{*}$.
For each $0\neq\alpha\in\mathbb{N}_{0}^{\infty}$, it follows from $e_\alpha\in H_0^p(Kdm_\infty)$ that
\begin{equation}\label{equ611}
\hat{K}(-\alpha)=\int_{\mathbb{T}^\infty}K e_\alpha dm_\infty=\int_{\mathbb{T}^\infty}1\cdot e_\alpha Kdm_\infty=0.
\end{equation}
Since $K$ is nonnegative, the conjugate of (\ref{equ611}) gives that
$$\hat{K}(\alpha)=\int_{\mathbb{T}^\infty}K \bar{e}_\alpha dm_\infty=0,$$
and the proof is complete.
\end{proof}

\section{Nevanlinna functions and Dirichlet series}

It is well known that there exists a fascinating connection between functions in infinitely many variables and Dirichlet series via Bohr correspondence. As in cases of the Nevanlinna class and the Smirnov class in the infinite-variable
setting, there exist analogies for Dirichlet series.
In this section, we will develop the Nevanlinna-Dirichlet class $\mathcal{N}$ and the Smirnov-Dirichlet class $\mathcal{N}_{*}$ for Dirichlet series.
Moreover, the relationships between $\mathcal{N}$ and $N(\mathbb{D}_1^\infty)$, $\mathcal{N}_{*}$ and $N_{*}(\mathbb{D}_1^\infty)$ will be established, respectively.

First, we briefly recall some elements from the theory of Dirichlet series.
A Dirichlet series is a series of the following form:
$$f(s)=\sum_{n=1}^\infty a_n n^{-s},$$
where $s$ is the complex variable.
Let $\mathbb{R}$ be the real line and $\overline{\mathbb{R}}=\mathbb{R}\cup\{\pm\infty\}$ the extended real line. When $\sigma\in\mathbb{R}$, set $\mathbb{C}_{\sigma}=\{z\in\mathbb{C}: \mathrm{Re}\;z>\sigma\}$. For each Dirichlet series $f(s)=\sum_{n=1}^\infty a_n n^{-s}$, its abscissa of uniform convergence $\sigma_{u}(f)$ is defined as
$$\sigma_{u}(f)=\inf\left\{\sigma\in\mathbb{R}: \sum_{n=1}^\infty a_n n^{-s}\;\text{converges uniformly on}\;\mathbb{C}_\sigma\right\}\in\overline{\mathbb{R}},$$
see \cite[pp.10]{DGMS}. When $\sigma\in\mathbb{R}$, write
$$f_\sigma(s)=\sum_{n=1}^\infty a_n n^{-(s+\sigma)},$$
then $\sigma_{u}(f_\sigma)=\sigma_{u}(f)-\sigma$.

Recall that a subset $S$ of $\mathbb{R}$ is relatively dense in $\mathbb{R}$, if there exists a constant $L>0$ such that each closed interval of length $L$ intersects $S$.
A function $f: \mathbb{R}\rightarrow\mathbb{C}$ is called almost periodic, if for each $\varepsilon>0$, there is a relatively dense set $E_{\varepsilon}\subset\mathbb{R}$ satisfying
$$\sup_{t\in\mathbb{R}}|f(t+\tau)-f(t)|<\varepsilon,\quad\forall\tau\in E_{\varepsilon}.$$
We mention that when $f$ is almost periodic, the limit $\lim_{T\rightarrow\infty}\frac{1}{2T}\int_{-T}^{T}f(t)dt$ exists. Let $f$ be a Dirichlet series with $\sigma_{u}(f)<0$, then $t\mapsto f(it)$ is almost periodic. Furthermore, if $h$ is a non-decreasing continuous function on the half real line $\mathbb{R}^{+}=[0,\infty)$, then $h(|f(it)|$ is  almost periodic, see \cite{Bes, QQ}.

Let $p_j$ be the $j$-th prime number. For every $n\in\mathbb{N}$, there exists a unique prime factorization $n=p_{1}^{\alpha_1}\cdots p_{k}^{\alpha_k}$ and set $\alpha(n)=(\alpha_1,\ldots,\alpha_k,0,\ldots)$. When $\zeta=(\zeta_1,\zeta_2,\ldots)$ is a sequence of complex numbers, we write
$\zeta^{\alpha(n)}=\zeta_{1}^{\alpha_1}\cdots\zeta_{k}^{\alpha_k}$. From Bohr's point of view \cite{Boh}, each Dirichlet series $f(s)=\sum_{n=1}^\infty a_n n^{-s}$ can be associated with a formal power series in infinitely many variables as follows: $$(\mathcal{B}f)(\zeta)=\sum_{n=1}^\infty a_n \zeta^{\alpha(n)}.$$
It is easy to verify that if $\sigma_{u}(f)<0$, then the partial sums of $\mathcal{B}f$ converge uniformly on $\mathbb{T}^\infty$, and $\mathcal{B}f\in A(\mathbb{T}^\infty)$, see \cite{HLS, QQ}.

This section contains two parts. In the first part, we define the Nevanlinna-Dirichlet class $\mathcal{N}$, and discuss the relationship between this class and the Nevanlinna class $N(\mathbb{D}_{1}^\infty)$. The main result is Theorem \ref{thm3101}. The second part is devoted to proving this theorem.

\subsection{The Nevanlinna-Dirichlet class $\mathcal{N}$}

Before defining the Nevanlinna-Dirichlet class $\mathcal{N}$, let us start with a smaller class, the Smirnov-Dirichlet class $\mathcal{N}_{*}$.

Let $f$ be a Dirichlet series with $\sigma_u(f)<0$, then $t\mapsto\log(1+|f(it)|)$ is almost periodic, and hence the limit
$$\|f\|_0=\lim_{T\rightarrow\infty}\frac{1}{2T}\int_{-T}^{T}\log(1+|f(it)|)dt$$
exists.
Note that $\mathcal{B}f\in A(\mathbb{T}^\infty)$. Applying Birkhoff-Oxtoby theorem \cite[Theorem 6.5.1]{QQ} yields
\begin{equation}\label{equ3101}
\|f\|_0=\int_{\mathbb{T}^\infty}\log(1+|\mathcal{B}f|)dm_\infty=\|\mathcal{B}f\|_0.
\end{equation}
Let $\mathcal{P}_D$ denote the set of all Dirichlet polynimials, then for every $Q\in\mathcal{P}_D$, $\sigma_u(Q)=-\infty$. Hence $\|Q\|_0$ is well defined and $\|Q\|_0=\|\mathcal{B}Q\|_0$.
We define the Smirnov-Dirichlet class $\mathcal{N}_{*}$ to be the completion of $\mathcal{P}_D$ in the metric $\|\cdot\|_0$.
Then there exists a natural algebraic structure on $\mathcal{N}_{*}$.
When $0<p<\infty$, there is a constant $C_p>0$ such that
$$\|Q\|_0\leq C_p\|Q\|_p^{\min\{p,1\}},\quad Q\in\mathcal{P}_D.$$
Hence all Hardy-Dirichlet spaces $\mathcal{H}^{p}\;(0<p\leq\infty)$ are contained in $\mathcal{N}_{*}$. Since $\mathcal{P}_\infty$ is dense in $N_{*}(\mathbb{T}^\infty)$, the Bohr transform $\mathcal{B}: \mathcal{P}_D\rightarrow\mathcal{P}_\infty$ can be extended to an isometric isomorphism from $\mathcal{N}_{*}$ onto $N_{*}(\mathbb{T}^\infty)$, still denoted by $\mathcal{B}$.
Furthermore, $\mathcal{B}$ is also an algebra isomorphism.

For a fixed $k\in\mathbb{N}$, let $\Xi_k$ be the multiplicative subsemigroup of $\mathbb{N}$ generated by the first $k$ prime numbers $p_1,\ldots,p_k$.
When $f$ is a Dirichlet series of the form $f(s)=\sum_{n\in\Xi_k}a_n n^{-s}$, then $(\mathcal{B}f)(\zeta)=\sum_{n\in\Xi_k}a_n\zeta^{\alpha(n)}$ only depends on the first $k$ variables $\zeta_1,\ldots,\zeta_k$. If $\mathcal{B}f$ is holomorphic on $\mathbb{D}^{k}$,
then  $f(s)=(\mathcal{B}f)(p_{1}^{-s},\ldots,p_{k}^{-s})$ converges uniformly on each $\mathbb{C}_\sigma\;(\sigma>0)$, and thus $\sigma_{u}(f)\leq0$. This inspires us to consider the class $\mathcal{N}_u$ of Dirichlet series $f$ with $\sigma_{u}(f)\leq0$ and $\limsup_{\sigma\rightarrow0^{+}}\|f_\sigma\|_0<\infty$.
By (\ref{equ2102B}), this coincides with Brevig and Perfekt's definition \cite{BP} mentioned in Introduction.
A similar argument as in the proof of Proposition \ref{thm2101} gives that if $f$ is a Dirichlet series with $\sigma_{u}(f)\leq0$, then $\|f_\sigma\|_0$ defines a non-increasing function of $\sigma$, which is also mentioned in \cite{BP}. Thus we have
$$\limsup_{\sigma\rightarrow0^{+}}\|f_\sigma\|_0=\sup_{\sigma>0}\|f_\sigma\|_0.$$

It is clear that $\mathcal{N}_u$ is a linear space. We define
$$\|f\|_0=\sup_{\sigma>0}\|f_\sigma\|_0,\quad f\in\mathcal{N}_u.$$
Then
$$d_0(f,g)=\|f-g\|_{0},\quad f,g\in\mathcal{N}_{u}$$
defines a metric on $\mathcal{N}_u$.
Unfortunately, the class $\mathcal{N}_{u}$ does not contain the familiar Hardy-Dirichlet space $\mathcal{H}^{2}$.
For example, put $a_n=(\sqrt{n}\log n)^{-1}$, then the Dirichlet series $f_{\boldsymbol{a}}(s)=\sum_{n=2}^\infty a_n n^{-s}$ belongs to $\mathcal{H}^{2}$, and $\sigma_{u}(f_{\boldsymbol{a}})=\frac{1}{2}$.
Hence $f_{\boldsymbol{a}}\notin\mathcal{N}_{u}$.
So we define the Nevanlinna-Dirichlet class $\mathcal{N}$ to be the completion of $\mathcal{N}_{u}$ in the metric $\|\cdot\|_0$.
Since $\mathcal{P}_D$ is contained in $\mathcal{N}_u$, we see that $\mathcal{N}_{*}$ is contained in $\mathcal{N}$.
In conclusion, we have
$$\mathcal{H}^{\infty}\subset\mathcal{H}^{q}\subset\mathcal{H}^{p}\subset\mathcal{N}_{*}\subset\mathcal{N}\quad(0<p<q<\infty).$$
The following is the main result of this section, which will be proved in Subsection 5.2.

\begin{thm}\label{thm3101}
Let $f(s)=\sum_{n=1}^\infty a_n n^{-s}$ be a Dirichlet series in $\mathcal{N}_u$. Then the formal power series $(\mathcal{B}f)(\zeta)=\sum_{n=1}^\infty a_n \zeta^{\alpha(n)}$ converges in $\mathbb{D}_{1}^\infty$. Moreover, we have $\mathcal{B}f\in N(\mathbb{D}_{1}^\infty)$, and $\|\mathcal{B}f\|_0=\|f\|_0$. Therefore, the Bohr transform $\mathcal{B}$ can be extended to an isometry from $\mathcal{N}$ to $N(\mathbb{D}_{1}^\infty)$, still denoted by $\mathcal{B}$.
\end{thm}

For a general Dirichlet series $f$ with $\sigma_u(f)\leq0$, the conclusion of Theorem \ref{thm3101} fails. Here is an example.

\begin{exam}\label{thm3101B}
Let $p_k$ be the $k$-th prime number and $\{p_{k_j}\}_{j=1}^\infty$ a subsequence of $\{p_k\}_{k=1}^\infty$ satisfying $p_{k_{j+1}}>2p_{k_j}$ for all $j\in\mathbb{N}$. Write
$$f(s)=\sum_{j=1}^\infty\frac{\log p_{k_j}}{p_{k_j}^s}.$$
Let $\sigma>0$, then there is a constant $M_\sigma>0$ such that for each $j\in\mathbb{N}$, $\log p_{k_j}\leq M_\sigma p_{k_j}^\sigma$. Therefore,
$$\sum_{j=1}^\infty\frac{\log p_{k_j}}{p_{k_j}^{2\sigma}}\leq M_\sigma\sum_{j=1}^\infty p_{k_j}^{-\sigma}\leq M_\sigma p_{k_1}^{-\sigma}\sum_{j=1}^\infty 2^{-(j-1)\sigma}<\infty.$$
Since $\sigma$ is arbitrary, we see that $\sigma_{u}(f)\leq0$. However, the sequence $\{\log p_{k_j}\}_{j=1}^\infty$ is unbounded, which implies that
$(\mathcal{B}f)(\zeta)=\sum_{j=1}^\infty(\log p_{k_j})\zeta_{k_j}$ diverges for some $\zeta\in\mathbb{D}_{1}^\infty$.
\end{exam}

The following example implies that $\mathcal{N}_{*}$ is a proper subclass of $\mathcal{N}$.

\begin{exam}
Write $\varphi(z)=\exp\left(\frac{1+z}{1-z}\right)$, $z\in\mathbb{D}$. It is easy to see that $\varphi\in N(\mathbb{D})$, but $\varphi\notin N_{*}(\mathbb{D})$. Let $a_n$ denote the $n$-th Taylor's coefficient of $\varphi$, and put $f(s)=\sum_{n=0}^\infty a_n 2^{-ns}$. Then for every $\zeta\in\mathbb{D}_{1}^\infty$, we have $(\mathcal{B}f)(\zeta)=\varphi(\zeta_1)$, and hence $\mathcal{B}f\notin N_{*}(\mathbb{D}_{1}^\infty)$. Note that $\mathcal{B}f$ is a holomorphic function on $\mathbb{D}$. As we mentioned before, $\sigma_{u}(f)\leq0$. On the other hand, it follows from (\ref{equ3101}) that
$$\sup_{\sigma>0}\|f_\sigma\|_0=\sup_{\sigma>0}\|\mathcal{B}f_\sigma\|_0
=\sup_{\sigma>0}\int_{\mathbb{T}}\log(1+|\varphi(2^{-\sigma}\lambda)|)dm_1(\lambda)<\infty,$$
which implies $f\in\mathcal{N}_{u}$.
Now we claim that $f\notin\mathcal{N}_{*}$.
Indeed, if $f\in\mathcal{N}_{*}$, then there exists a sequence of Dirichlet polynomials $\{Q_n\}_{n=1}^\infty$ such that $\|Q_n-f\|_0\rightarrow0$ as $n\rightarrow\infty$. By Theorem \ref{thm3101}, $\|\mathcal{B}Q_n-\mathcal{B}f\|_0\rightarrow0$ as $n\rightarrow\infty$, which gives that $\mathcal{B}f\in N_{*}(\mathbb{D}_{1}^\infty)$, a contradiction.
\end{exam}

Let $0<p<q<\infty$. We draw up the following figure to show relations between all spaces, where ``$\simeq$" denotes that the corresponding spaces are canonically isometrically isomorphic, and ``$\hookrightarrow$" denotes that the former space can be isometrically embedded into the latter.
It remains to be clarified whether $\mathcal{B}: \mathcal{N}\rightarrow N(\mathbb{D}_{1}^\infty)$ is surjective.

\begin{center}
\tabcolsep0.05in
\renewcommand\arraystretch{1.5}
\begin{tabular}{ccccccccc}
$\mathcal{H}^{\infty}$&$\subset$&$\mathcal{H}^{q}$&$\subset$&$\mathcal{H}^{p}$&$\subset$&$\mathcal{N}_{*}$&$\subset$&$\mathcal{N}$\\
\rotatebox[origin = c]{270}{$\simeq$}&\quad&\rotatebox[origin = c]{270}{$\simeq$}&\quad&\rotatebox[origin = c]{270}{$\simeq$}&\quad&\rotatebox[origin = c]{270}{$\simeq$}&\quad&\rotatebox[origin = c]{270}{$\hookrightarrow$}\\
$H^{\infty}(\mathbb{D}_2^\infty)$&$\subset$&$H^{q}(\mathbb{D}_2^\infty)$&$\subset$&$H^{p}(\mathbb{D}_2^\infty)$&$\subset$&$N_{*}(\mathbb{D}_{1}^\infty)$&$\subset$&$N(\mathbb{D}_{1}^\infty)$\\
\end{tabular}
\end{center}

\subsection{Proof of Theorem \ref{thm3101}}

This subsection is mainly dedicated to proving Theorem \ref{thm3101}. To prove this theorem, a series of preparations is needed.

Let $F$ be a holomorphic function on the polydisk $\mathbb{D}^k$. For each $\sigma>0$, write
$$F_{\{\sigma\}}(w)=F(p_1^{-\sigma}w_1,\ldots,p_k^{-\sigma}w_k),\quad w\in\mathbb{T}^k,$$
where $p_j$ denotes the $j$-th prime number. Then by the subharmonicity of $\log(1+|F|)$ in each variable separately,
\begin{equation}\label{equ3200}
\sup_{\sigma>0}\int_{\mathbb{T}^k}\log(1+|F_{\{\sigma\}}|)dm_k=\sup_{0<r<1}\int_{\mathbb{T}^k}\log(1+|F_{[r]}|)dm_k,
\end{equation}
where $m_k$ denotes the normalized Lebesgue measure on $\mathbb{T}^k$. Moreover, integrals in the left side of this equality decrease with $\sigma$.

Suppose that $f(s)=\sum_{n=1}^\infty a_n n^{-s}$ is a Dirichlet series in $\mathcal{N}_u$. We first consider Bohr's $k$te Abschnitt
$$(\mathcal{B}_{k}f)(\zeta)=\sum_{n\in\Xi_k}a_n\zeta^{\alpha(n)}$$
of $\mathcal{B}f$.
It is a formal power series only depending on the first $k$ variables $\zeta_1,\ldots,\zeta_k$. We have the following proposition.

\begin{prop}\label{thm3201}
If $f\in\mathcal{N}_u$, then for every $k\in\mathbb{N}$, the formal power series $\mathcal{B}_k f$ converges in $\mathbb{D}^k$, and it defines a function in the Nevanlinna class $N(\mathbb{D}^k)$ over $\mathbb{D}^k$.
\end{prop}

\begin{proof}
Let $f(s)=\sum_{n=1}^\infty a_n n^{-s}$ be a Dirichlet series in $\mathcal{N}_u$, then for every $\sigma>0$, $\sigma_{u}(f_\sigma)<0$, and thus $\mathcal{B}f_{\sigma}\in A(\mathbb{T}^\infty)$. Applying Lebesgue's dominated convergence theorem gives
\begin{equation}\label{equ3201}
\lim_{k\rightarrow\infty}\int_{\mathbb{T}^\infty}\log(1+|\mathcal{B}_k f_\sigma|)dm_\infty=\int_{\mathbb{T}^\infty}\log(1+|\mathcal{B}f_\sigma|)dm_\infty=\|\mathcal{B}f_\sigma\|_0.
\end{equation}
It follows from Corollary \ref{thm2101B} that the integrals in the left side of (\ref{equ3201}) increase with $k$. Therefore we have
$$\sup_{k\in\mathbb{N}}\sup_{\sigma>0}\int_{\mathbb{T}^\infty}\log(1+|\mathcal{B}_k f_{\sigma}|)dm_\infty=
\sup_{\sigma>0}\sup_{k\in\mathbb{N}}\int_{\mathbb{T}^\infty}\log(1+|\mathcal{B}_k f_{\sigma}|)dm_\infty=\sup_{\sigma>0}\|\mathcal{B}f_\sigma\|_0.$$
By (\ref{equ3101}), for every $\sigma>0$, $\|\mathcal{B}f_\sigma\|_0=\|f_\sigma\|_0$, and hence
\begin{equation}\label{equ3202}
\sup_{k\in\mathbb{N}}\sup_{\sigma>0}\int_{\mathbb{T}^\infty}\log(1+|\mathcal{B}_k f_{\sigma}|)dm_\infty=\sup_{\sigma>0}\|\mathcal{B}f_\sigma\|_0=
\sup_{\sigma>0}\|f_\sigma\|_0=\|f\|_0<\infty.
\end{equation}
For a fixed $k\in\mathbb{N}$, the family $\{\mathcal{B}_k f_\sigma\}_{\sigma>0}$ is contained in the polydisk algebra $A(\mathbb{D}^{k})$. It follows from Lemma \ref{thm2105B} that for each $\zeta\in\mathbb{D}^{k}$,
$$\log(1+|(\mathcal{B}_k f_\sigma)(\zeta)|)\leq\|\mathbf{P}_\zeta\|_\infty\int_{\mathbb{T}^\infty}\log(1+|\mathcal{B}_k f_{\sigma}|)dm_\infty\leq
\|\mathbf{P}_\zeta\|_\infty\|f\|_0.$$
This implies that $\{\mathcal{B}_k f_\sigma\}_{\sigma>0}$ is uniformly bounded on each compact subset of $\mathbb{D}^{k}$. By Montel's theorem \cite[Theorem 1.5]{Ohs}, there is a sequence $\sigma_m \rightarrow0^{+}\;(m\rightarrow\infty)$ and a function $G_k$ holomorphic on $\mathbb{D}^{k}$ such that $\{\mathcal{B}_k f_{\sigma_m}\}_{m=1}^\infty$ converges to $G_k$ uniformly on each compact subset of $\mathbb{D}^{k}$. Let
$$G_{k}(\zeta)=\sum_{n\in\Xi_k} b_n \zeta^{\alpha(n)},\quad \zeta\in\mathbb{D}^{k}$$
be the Taylor expansion of $G_k$. Since for every $m\in\mathbb{N}$,
$$(\mathcal{B}_k f_{\sigma_m})(\zeta)=\sum_{n\in\Xi_k}a_n n^{-\sigma_m}\zeta^{\alpha(n)},\quad \zeta\in\mathbb{D}^{k},$$
the uniform convergence implies that
$$b_n=\lim_{m\rightarrow\infty}a_n n^{-\sigma_m}=a_n,\quad n\in\Xi_k,$$
and hence $(\mathcal{B}_k f)(\zeta)=\sum_{n\in\Xi_k}a_n\zeta^{\alpha(n)}$ converges in $\mathbb{D}^k$. Noticing that for every $\sigma>0$, $\mathcal{B}_k f_\sigma=(\mathcal{B}_k f)_{\{\sigma\}}$, and combining (\ref{equ3200}) with (\ref{equ3202}) show that
\begin{equation}\label{equ3203}
\sup_{0<r<1}\int_{\mathbb{T}^\infty}\log(1+|(\mathcal{B}_k f)_{[r]}|)dm_\infty=\sup_{\sigma>0}\int_{\mathbb{T}^\infty}\log(1+|\mathcal{B}_k f_\sigma|)dm_\infty\leq\|f\|_0<\infty,
\end{equation}
which implies $\mathcal{B}_k f\in N(\mathbb{D}^k)$.
\end{proof}

For every $k\in\mathbb{N}$, $N(\mathbb{D}^k)$ can be regarded as a subset of $N(\mathbb{D}_1^\infty)$.
Hence by Proposition \ref{thm3201}, when $f\in\mathcal{N}_u$, $\{\mathcal{B}_k f\}_{k=1}^\infty$ is a sequence in $N(\mathbb{D}_1^\infty)$.
Then it follows from Lemma \ref{thm2105B} and (\ref{equ3203}) that for each $\zeta\in\mathbb{D}_1^\infty$,
$$\log(1+|(\mathcal{B}_k f)(\zeta)|)\leq\|\mathbf{P}_\zeta\|_\infty\|\mathcal{B}_k f\|_0\leq\|\mathbf{P}_\zeta\|_\infty\|f\|_0.$$
Therefore we have

\begin{cor}\label{thm3202B}
If $f\in\mathcal{N}_u$, then the sequence $\{\mathcal{B}_k f\}_{k=1}^\infty$ is uniformly bounded on each domain $V_{r,M}\;(0<r<1, M>0)$ defined by (\ref{equ2103A}).
\end{cor}

We need the following result.

\begin{prop}\label{thm3202}
If $f\in\mathcal{N}_u$, then the sequence $\{\mathcal{B}_k f\}_{k=1}^\infty$ converges pointwise on $\mathbb{D}_1^\infty$, and
$$(\hat{\mathcal{B}}f)(\zeta)=\lim_{k\rightarrow\infty}(\mathcal{B}_k f)(\zeta),\quad\zeta\in\mathbb{D}_1^\infty$$
defines a holomorphic function on $\mathbb{D}_1^\infty$.
\end{prop}

The following lemma is used in the proof of Proposition \ref{thm3202}.

\begin{lem}\label{thm3203}
If $\{a_n\}_{n=1}^\infty\in\ell^1$, then there exist $\{b_n\}_{n=1}^\infty \in\ell^1$ and $\{c_n\}_{n=1}^\infty\in c_0$ such that $a_n=b_n c_n$, $n=1,2,\ldots$, where $c_0$ denotes the Banach space of null sequences.
\end{lem}

\begin{proof}
Without loss of generality, we may assume that $\{a_n\}_{n=1}^\infty$ has infinitely many nonzero entries. It suffices to show that there is a sequence $\{\lambda_n\}_{n=1}^\infty$ such that $\lambda_n\rightarrow\infty$ as $n\rightarrow\infty$, and $\sum_{n=1}^\infty |a_n \lambda_n|<\infty$. Write $A=\sum_{n=1}^\infty|a_n|$ and
$$k_j=\min\left\{k\in\mathbb{N}: \sum_{n=1}^k |a_n|\geq\frac{6A}{\pi^2}\sum_{n=1}^j \frac{1}{n^2}\right\},\quad j\in\mathbb{N}.$$
Then there is a subsequence $\{k_{s_j}\}_{j=1}^\infty$ of $\{k_{j}\}_{j=1}^\infty$ such that $k_{s_1}<k_{s_2}<\cdots$, and for each $j\in\mathbb{N}$,  we have
$$\sum_{n=1}^{k_{s_j}}|a_n|\geq\frac{6A}{\pi^2}\sum_{n=1}^{s_j}\frac{1}{n^2},\quad\sum_{n=1}^{k_{s_{j+1}}}|a_n|
\leq\sum_{n=1}^{k_{s_{j+2}}-1}|a_n|\leq\frac{6A}{\pi^2}\sum_{n=1}^{s_{j+2}}\frac{1}{n^2},$$
and hence
\begin{equation}\label{equ3204}
\sum_{n=k_{s_j}+1}^{k_{s_{j+1}}}|a_n|\leq\frac{6A}{\pi^2}\sum_{n=s_j +1}^{s_{j+2}}\frac{1}{n^2}.
\end{equation}
If $k_{s_j}+1\leq n\leq k_{s_{j+1}}$, set $\lambda_n=\sqrt{s_j +1}$. Then $\lambda_n\rightarrow\infty$ as $n\rightarrow\infty$, and by (\ref{equ3204}),
$$\begin{aligned}
\sum_{n=k_{s_1}+1}^{\infty}|a_n \lambda_n|&=\sum_{j=1}^\infty\sum_{n=k_{s_j}+1}^{k_{s_{j+1}}}|a_n \lambda_n|\\
&=\sum_{j=1}^\infty\sqrt{s_j +1}\sum_{n=k_{s_j}+1}^{k_{s_{j+1}}}|a_n|\\
&\leq\frac{6A}{\pi^2}\sum_{j=1}^\infty\sqrt{s_j +1}\sum_{n=s_j +1}^{s_{j+2}}\frac{1}{n^2}\\
&\leq\frac{6A}{\pi^2}\sum_{j=1}^\infty\sum_{n=s_j +1}^{s_{j+2}}\frac{1}{n\sqrt{n}}\\
&\leq\frac{12A}{\pi^2}\sum_{n=1}^{\infty}\frac{1}{n\sqrt{n}}\\
&<\infty,
\end{aligned}$$
which completes the proof.
\end{proof}

We now present the proof of Proposition \ref{thm3202}.

\vskip2mm

\noindent\textbf{Proof of Proposition \ref{thm3202}.} We first show that $\{\mathcal{B}_k f\}_{k=1}^\infty$ converges pointwise on $\mathbb{D}_1^\infty$. For a fixed $\zeta\in\mathbb{D}_1^\infty$, it follows from Lemma \ref{thm3203} that there exist $\alpha\in\mathbb{D}_1^\infty$ and $\beta\in\mathbb{B}_{0}$ such that $\zeta_n=\alpha_n \beta_n$, $n=1,2,\ldots$, where $\mathbb{B}_{0}$ denotes the open unit ball of $c_0$. By Corollary \ref{thm3202B}, $\{\mathcal{B}_k f\}_{k=1}^\infty$ is uniformly bounded on $\Delta_{\alpha}=\alpha_1\overline{\mathbb{D}}\times\cdots\times\alpha_n\overline{\mathbb{D}}\times\cdots$.
Set
$$C_\alpha=\sup_{k\in\mathbb{N},\upsilon\in \Delta_{\alpha}}|(\mathcal{B}_k f)(\upsilon)|+1.$$
For any $\varepsilon>0$, there exists an index $K$ for which if $k>K$, then $|\beta_k|<\frac{\varepsilon}{2C_\alpha}$.
By a similar argument as in \cite{Hil, HLS}, we conclude that $\{(\mathcal{B}_k f)(\zeta)\}_{k=1}^\infty$ is Cauchy.
Indeed, for $k>l>K$, set
$$F(z)=(\mathcal{B}_k f)(\alpha_1 \beta_1,\ldots,\alpha_l \beta_l,\alpha_{l+1}z_{1},\ldots,\alpha_{k}z_{k-l}),\quad z=(z_1,\ldots,z_{k-l})\in\mathbb{D}^{k-l}.$$
We see that $F$ is holomorphic on $\mathbb{D}^{k-l}$, and $\sup_{z\in\mathbb{D}^{k-l}}|F(z)|\leq C_\alpha$. By Schwarz's lemma \cite[Theorem 1.9]{Ohs}, we have
$$|(\mathcal{B}_k f)(\zeta)-(\mathcal{B}_l f)(\zeta)|=|F(\beta_{l+1},\ldots,\beta_k)-F(0)|\leq2C_\alpha\max_{l+1\leq j\leq k}|\beta_j|<\varepsilon,$$
which implies that $\{(\mathcal{B}_k f)(\zeta)\}_{k=1}^\infty$ is Cauchy, and hence $\{\mathcal{B}_k f\}_{k=1}^\infty$ converges pointwise on $\mathbb{D}_1^\infty$. Write
\begin{equation}\label{equ3205}
(\hat{\mathcal{B}}f)(\zeta)=\lim_{k\rightarrow\infty}(\mathcal{B}_k f)(\zeta),\quad\zeta\in\mathbb{D}_1^\infty.
\end{equation}
Then Corollary \ref{thm3202B} implies that $\hat{\mathcal{B}}f$ is locally bounded. In what follows we show that $\hat{\mathcal{B}}f$ is holomorphic on $\mathbb{D}_1^\infty$. For fixed $\zeta\in\mathbb{D}_1^\infty$ and $\xi\in\ell^{1}$, there exists an open domain $\Omega\subset\mathbb{C}$ such that $0\in\Omega$ and $\{\zeta+\lambda\xi: \lambda\in\Omega\}\subset V_{r,M}$ for some $0<r<1$, $M>0$. For each $k\in\mathbb{N}$, let
$$G_k(\lambda)=(\mathcal{B}_k f)(\zeta+\lambda\xi),\quad\lambda\in\Omega.$$
Then $\{G_k\}_{k=1}^\infty$ is sequence of holomorphic functions on $\Omega$, and by Corollary \ref{thm3202B}, it is uniformly bounded on $\Omega$. Applying Montel's theorem, there is a subsequence that converges uniformly on each compact subset of $\Omega$. By combining this with (\ref{equ3205}), $(\hat{\mathcal{B}}f)(\zeta+\lambda\xi)$ is holomorphic in parameter $\lambda\in\Omega$, and hence $\hat{\mathcal{B}}f$ is holomorphic on $\mathbb{D}_1^\infty$. $\hfill\square$

\vskip2mm

We mention that every function $F$ holomorphic on $\mathbb{D}_{1}^\infty$ has a unique monomial expansion
\begin{equation}\label{equ2104A}
F(\zeta)=\sum_{n=1}^\infty c_n \zeta^{\alpha(n)},
\end{equation}
and the series converges uniformly and absolutely on compact subsets of $\mathbb{D}_{1}^\infty$ for which  we refer readers to \cite{DMP}.
Now we present the proof of Theorem \ref{thm3101}.

\vskip2mm

\noindent\textbf{Proof of Theorem \ref{thm3101}.} We first prove that $(\mathcal{B}f)(\zeta)=\sum_{n=1}^\infty a_n \zeta^{\alpha(n)}$ converges to a holomorphic function on $\mathbb{D}_{1}^\infty$. By Proposition \ref{thm3202}, the sequence $\{\mathcal{B}_k f\}_{k=1}^\infty$ converges to a holomorphic function $\hat{\mathcal{B}}f$. Let
\begin{equation}\label{equ3206}
(\hat{\mathcal{B}}f)(\zeta)=\sum_{n=1}^\infty c_n \zeta^{\alpha(n)},\quad\zeta\in\mathbb{D}_{1}^\infty
\end{equation}
be the monomial expansion of $\hat{\mathcal{B}}f$. For a fixed $k\in\mathbb{N}$, Bohr's $k$te Abschnitt $A_k(\hat{\mathcal{B}}f)$ is holomorphic on $\mathbb{D}^k$, and
$$A_k(\hat{\mathcal{B}}f)(\zeta)=\sum_{n\in\Xi_k} c_n \zeta^{\alpha(n)},\quad\zeta\in\mathbb{D}^k.$$
On the other hand, by the definition (\ref{equ3205}) of $\hat{\mathcal{B}}f$, we have
$$A_k(\hat{\mathcal{B}}f)(\zeta)=(\mathcal{B}_k f)(\zeta)=\sum_{n\in\Xi_k} a_n \zeta^{\alpha(n)},\quad\zeta\in\mathbb{D}^k.$$
It follows from the uniqueness of Taylor expansion that for all $n\in\Xi_k$, $a_n=c_n$, and hence $a_n=c_n$ for all $n\in\mathbb{N}$ by the arbitrariness of $k$.
Comparison of this and (\ref{equ3206}) shows that
$$(\mathcal{B}f)(\zeta)=\sum_{n=1}^\infty a_n \zeta^{\alpha(n)}=\sum_{n=1}^\infty c_n \zeta^{\alpha(n)}$$
converges in $\mathbb{D}_{1}^\infty$, and $\mathcal{B}f=\hat{\mathcal{B}}f$.

We next show that $\mathcal{B}f\in N(\mathbb{D}_{1}^\infty)$ and $\|\mathcal{B}f\|_0=\|f\|_0$. Since $\mathcal{B}f$ is holomorphic on $\mathbb{D}_{1}^\infty$, $(\mathcal{B}f)_{[r]}\in A(\mathbb{T}^\infty)$ for all $0<r<1$, and thus
$$\|\mathcal{B}f\|_0=\sup_{0<r<1}\int_{\mathbb{T}^\infty}\log(1+|(\mathcal{B}f)_{[r]}|)dm_\infty
=\sup_{0<r<1}\sup_{k\in\mathbb{N}}\int_{\mathbb{T}^\infty}\log(1+|(\mathcal{B}_k f)_{[r]}|)dm_\infty.$$
By interchanging the order of taking supremum for $r$ and $k$, we have
\begin{equation}\label{equ3207}
\|\mathcal{B}f\|_0=\sup_{k\in\mathbb{N}}\sup_{0<r<1}\int_{\mathbb{T}^\infty}\log(1+|(\mathcal{B}_k f)_{[r]}|)dm_\infty.
\end{equation}
Since for $k\in\mathbb{N}$, $\mathcal{B}_k f$ is holomorphic on $\mathbb{D}^k$, it follows from (\ref{equ3200}) that
$$\sup_{0<r<1}\int_{\mathbb{T}^\infty}\log(1+|(\mathcal{B}_k f)_{[r]}|)dm_\infty=\sup_{\sigma>0}\int_{\mathbb{T}^\infty}\log(1+|\mathcal{B}_k f_\sigma|)dm_\infty.$$
Substituting this into (\ref{equ3207}) yields
\begin{equation}\label{equ3208}
\|\mathcal{B}f\|_0=\sup_{k\in\mathbb{N}}\sup_{\sigma>0}\int_{\mathbb{T}^\infty}\log(1+|\mathcal{B}_k f_\sigma|)dm_\infty=
\sup_{\sigma>0}\sup_{k\in\mathbb{N}}\int_{\mathbb{T}^\infty}\log(1+|\mathcal{B}_k f_\sigma|)dm_\infty.
\end{equation}
Note that for every $\sigma>0$, $\mathcal{B}f_{\sigma}\in A(\mathbb{T}^\infty)$. Then
$$\sup_{k\in\mathbb{N}}\int_{\mathbb{T}^\infty}\log(1+|\mathcal{B}_k f_\sigma|)dm_\infty=\int_{\mathbb{T}^\infty}\log(1+|\mathcal{B}f_\sigma|)dm_\infty=\|\mathcal{B}f_\sigma\|_0.$$
We conclude from this equality, (\ref{equ3101}) and (\ref{equ3208})  that
$$\|\mathcal{B}f\|_0=\sup_{\sigma>0}\|\mathcal{B}f_\sigma\|_0=\sup_{\sigma>0}\|f_\sigma\|_0=\|f\|_0<\infty.$$
The proof is complete. $\hfill\square$

\vskip2mm

Let $f(s)=\sum_{n=1}^\infty a_n n^{-s}$ be a Dirichlet series with $\sigma_u(f)\leq0$. Example \ref{thm3101B} shows that
$(\mathcal{B}f)(\zeta)=\sum_{n=1}^\infty a_n \zeta^{\alpha(n)}$
does not need to converge at all points in $\mathbb{D}_1^\infty$.
However, for every $0<r<1$, the power series
$$(\mathcal{B}_r f)(\zeta)=\sum_{n=1}^\infty a_n (r\zeta_1,\ldots,r^{k}\zeta_k,\ldots)^{\alpha(n)}$$
converges in $\overline{\mathbb{D}}^\infty$, and $\mathcal{B}_r f\in A(\mathbb{D}^\infty)$. Indeed, let $p_n$ be the $n$-th prime number, then the prime number theorem implies that $\frac{p_n}{n\log n}\rightarrow1$ as $n\rightarrow\infty$ \cite[Theorem 4.5]{Apo}. Therefore, for a fixed $0<r<1$, there exists $\sigma>0$ such that for all $n\in\mathbb{N}$, $r^n\leq p_n^{-\sigma}$, and hence
\begin{equation}\label{equ3209}
r\overline{\mathbb{D}}\times\cdots\times r^n\overline{\mathbb{D}}\times\cdots\subset p_1^{-\sigma}\overline{\mathbb{D}}\times\cdots\times p_n^{-\sigma}\overline{\mathbb{D}}\times\cdots.
\end{equation}
As mentioned in \cite{HLS}, the partial sums of $\mathcal{B}f_\sigma$ converge uniformly on $\overline{\mathbb{D}}^\infty$. That is, for any $\varepsilon>0$, there is an index $N$ for which if $M_2>M_1>N$, then
$$\left\|\sum_{n=M_1}^{M_2}a_n (p_1^{-\sigma}\zeta_1,\ldots,p_k^{-\sigma}\zeta_k,\ldots)^{\alpha(n)}\right\|_\infty<\varepsilon,$$
where $\|\cdot\|_\infty$ denotes the uniform norm in Banach algebra $C(\overline{\mathbb{D}}^\infty)$.
From (\ref{equ3209}), we have
$$\left\|\sum_{n=M_1}^{M_2}a_n (r\zeta_1,\ldots,r^{k}\zeta_k,\ldots)^{\alpha(n)}\right\|_\infty\leq\left\|\sum_{n=M_1}^{M_2}a_n (p_1^{-\sigma}\zeta_1,\ldots,p_k^{-\sigma}\zeta_k,\ldots)^{\alpha(n)}\right\|_\infty<\varepsilon,$$
and thus the partial sums of $\mathcal{B}_r f$ converge uniformly on $\overline{\mathbb{D}}^\infty$. This gives $\mathcal{B}_r f\in A(\mathbb{D}^\infty)$ as desired. By this fact and a similar argument as in the proof of Theorem \ref{thm3101}, we have the following conclusion.

\begin{prop}
Let $f$ be a Dirichlet series with $\sigma_u(f)\leq0$. Then
$$\sup_{\sigma>0}\|f_\sigma\|_0=\sup_{\sigma>0}\int_{\mathbb{T}^\infty}\log(1+|\mathcal{B}f_\sigma|)dm_\infty
=\sup_{0<r<1}\int_{\mathbb{T}^\infty}\log(1+|\mathcal{B}_r f|)dm_\infty.$$
\end{prop}

\noindent\textbf{Acknowledgements:}\quad This work was partially supported by NNSF of China, and NSF of
Shanghai (21ZR1404200).

\vskip3mm \noindent{Kunyu Guo, School of Mathematical Sciences, Fudan
University, Shanghai, 200433, China, E-mail: kyguo@fudan.edu.cn

\noindent Jiaqi Ni, School of Mathematical Sciences, Fudan
University, Shanghai, 200433, China, E-mail: jqni15@fudan.edu.cn

\noindent Qi Zhou, School of Mathematical Sciences, Fudan
University, Shanghai, 200433, China, E-mail: qzhou17@fudan.edu.cn

\end{document}